\documentclass[11pt]{article}
\usepackage{amsfonts}
\usepackage{amsmath}
\usepackage{amssymb}
\usepackage{amsthm}
\usepackage{graphicx}
\usepackage{fullpage}
\usepackage{enumerate}
\usepackage{color}
\newtheorem{theorem}{Theorem}[section]
\newtheorem{lemma}[theorem]{Lemma}

\newtheorem{proposition}[theorem]{Proposition}
\newtheorem{corollary}[theorem]{Corollary}
\newtheorem{remark}[theorem]{Remark}

\theoremstyle{definition}

\newtheorem{thmy}{Theorem}
\newenvironment{oldtheorem}{\stepcounter{thm}\begin{thmy}}{\end{thmy}}
\newtheorem*{note*}{Note}

\makeatletter
\newcommand{\bd}{\textnormal{bd}\,}

\newcommand{\R}{\mathbb R}
\newcommand{\Rn}{{\mathbb R}^n}

\newcommand{\Sn}{\mathbb{ S}^{n-1}}

\newcommand{\Sen}{{\mathbb S}^1}
\newcommand{\Ha}{{{\cal H}^{n-1}}}

\newcommand{\inter}{\textnormal{int}\,}
\newcommand{\cl}{\textnormal{cl}\,}

\begin{document}


\title{\bf On a non-homogeneous version of a problem of Firey}
\date{}
\medskip
\author{Christos Saroglou}
\maketitle

\begin{abstract}
We investigate the uniqueness for the Monge-Amp\`{e}re type equation 
\begin{equation}
\label{eq-abstract}
\textnormal{det}(u_{ij}+\delta_{ij}u)_{i,j=1}^{n-1}=G(u),\qquad\textnormal{on} \ \ \Sn,
\end{equation}
where $u$ is the restriction of the support function on the sphere $\Sn$ of a convex body that contains the origin in its interior and $G:(0,\infty)\to(0,\infty)$ is a continuous function.  The problem was initiated by Firey (1974) who, in the case $G(\theta)=\theta^{-1}$, asked if $u\equiv 1$ is the unique solution to (1). Recently, Brendle, Choi and Daskalopoulos \cite{brendle_choi_daskalopoulos_2017} proved that if $G(\theta)=\theta^{-p}$, $p>-n-1$, then $u$ has to be constant, providing in particular a complete solution to Firey's problem. Our primary goal is to obtain uniqueness (or nearly uniqueness) results for (\ref{eq-abstract}) for a broader family of functions $G$. Our approach is very different than the techniques developed in \cite{brendle_choi_daskalopoulos_2017}.
\end{abstract}

\section{Introduction}
The primary goal of this note is to obtain uniqueness results for the Monge-Amp\`{e}re type equation
\begin{equation}\label{eq-monge-ampere}
\textnormal{det}(u_{ij}+\delta_{ij}u)_{i,j=1}^{n-1}=G(u),\qquad\textnormal{on} \ \ \Sn,
\end{equation}
where $G:(0,\infty)\to\R$ is a strictly positive continuous function and $u$ is the restriction on $\Sn$ of a sub-linear positively homogeneous function defined on $\Rn$. Here, $\Sn$ is the Euclidean unit sphere of $\Rn$, $u_{ij}$ denotes the covariant derivative of $u$ with respect to a local orthonormal frame on $\Sn$ and $\delta_{ij}$ are the Kronecker symbols.    

Let $K$ be a convex body (that is, convex compact with non-empty interior) in $\Rn$. The \emph{support function} $h_K:\Rn\to\R$ of $K$ is defined by
$$h_K(x)=\max\{\langle x,y\rangle:y\in K\}.$$ 
Recall (see e.g. \cite[Chapter 1]{schneider_2013}) that a function $u:\Rn\to\R$ is sub-linear and positively homogeneous if and only if there exists a (unique) convex body $K$ in $\Rn$, such that $u=h_K$. Furthermore, $h_K$ is strictly positive on $\Rn\setminus\{o\}$ if and only if $K$ contains $o$ (the origin) in its interior.

The \emph{surface area measure} $S_K$ of $K$ is a Borel measure on $\Sn$, given by
$$S_K(\omega)=\Ha(\{x:x\textnormal{ is a boundary point of }K\textnormal{ and there exists }v\in\omega,\textnormal{ such that } \langle x,v\rangle=h_K(v)\}),$$
for any Borel subset $\omega$ of $\Sn$. The notation $\Ha(\cdot)$ stands for the $(n-1)$-dimensional Hausdorff measure in $\Rn$. It is clear that $S_K$ is invariant under translation of $K$. Moreover, the barycentre of $S_K$ is always at the origin. If $S_K$ is absolutely continuous with respect to $\Ha$, its density will be denoted by $f_K$. It is well known that 
\begin{equation}\label{eq-F_K}
f_K(v)=\textnormal{det}\left(h_{K,ij}(v)+\delta_{ij}h_K(v)\right),\qquad\textnormal{for almost every }v \in\Sn .
\end{equation}
If the boundary of $K$ is of class $C^2$ and its curvature ${\cal K}(K,\cdot)$ is strictly positive, then $f_K$ is continuous (i.e. has a continuous representative) and it holds
\begin{equation}\label{eq-curvature}
f_K(v)=\frac{1}{{\cal K}(K,\eta_K^{-1}(v))},\qquad\textnormal{for all }v\in\Sn .
\end{equation}
Here, $\eta_K$ is the Gauss map (see Section 2). We refer to \cite{schneider_2013} for concepts related to (\ref{eq-F_K}), (\ref{eq-curvature}) and the definition of the surface area measure.

As (\ref{eq-F_K}) shows, (\ref{eq-monge-ampere}) can be re-written as follows
\begin{equation}\label{eq-thm-main}
dS_K=G(h_K)d\Ha.
\end{equation}
Equation (\ref{eq-monge-ampere}) and its weak form (\ref{eq-thm-main}) have appeared in several places in literature (in the context of different areas of Mathematics, such as Convex Geometry, Differential Geometry and PDE's). Below, we list some of them.

A result due to Simon \cite{simon_1967} states the following.
\begin{oldtheorem}\label{thm-old-Simon}(Simon \cite{simon_1967})
Let $u:\Sn\to(0,\infty)$ be the restriction of a support function on $\Sn$, $k\in\{1,\dots,n-1\}$, $\lambda_k(u)$ be the $k$-th elementary symmetric polynomial of the reciprocals of the eigenvalues of the matrix $\textnormal{det}(u_{ij}+\delta_{ij}u)_{i,j=1}^{n-1}$ and $G$ be a strictly positive non-decreasing function. If $u$ satisfies 
$$\lambda_k(u)=G(u),\qquad\textnormal{on }\ \Sn,$$then $u$ is constant.
\end{oldtheorem}
\noindent Theorem \ref{thm-old-Simon}, then, corresponds to (\ref{eq-monge-ampere}) if $k=n-1$.


Firey \cite{firey_1974} posed the following problem: Assume that $G(\theta)=\theta^{-1}$. Is it true that the unit ball is the only convex body $K$ that solves (\ref{eq-thm-main})? He noticed that the answer to the problem is affirmative if one additionally assumes that $b(K)=o$, where $b(K)$ is the barycentre of $K$. This follows as an almost immediate consequence of the Blaschke-Santal\'{o} inequality (see next section for notation)
$$V(K)V(K^\circ)\leq V(B_2^n)^2.$$
Andrews considered a generalization of Firey's problem by replacing $\theta^{-1}$ with $\theta^{p}$, $p\in\R\setminus\{0\}$. It turns out  that uniqueness fails if $p\leq -n-1$ (see \cite{andrews_2000} and \cite{andrews_2002} for a complete classification of solutions in the plane). The case $p=-n-1$ was settled in \cite{petty_1985}; it was shown that, in this case, solutions are precisely the support functions of ellipsoids centered at the origin.  The problem of classifying the solutions of (\ref{eq-thm-main}) in the case $G(\theta)=\theta^{p}$, $p\in\R\setminus\{0\}$, is closely related to the asymptotic behaviour of the solution to the $a$-\emph{Gauss curvature flow} (also introduced by Firey in \cite{firey_1974} for $a=1$ and by Andrews in \cite{andrews_2000} and \cite{andrews_2002} for general $a$), where $a=-1/p$, and has attracted considerable attention in the past decades.    

Brendle, Choi and Daskalopoulos \cite{brendle_choi_daskalopoulos_2017} recently solved the problem for $p>-n-1$ and, hence, gave an affirmative answer to Firey's long standing question.

\begin{oldtheorem}\label{thm-old-BCD}(Brendle, Choi, Daskalopoulos \cite{brendle_choi_daskalopoulos_2017})
If $K$ is a convex body, containing the origin in its interior, that solves (\ref{eq-thm-main}) for $G(\theta)=\theta^p$, $\theta>0$, $p>-n-1$, then $K$ is a Euclidean ball centered at the origin. 
\end{oldtheorem}
Intermediate and related results to the preceding theorem include (but certainly not limited to) \cite{Tso_1985, chow_1985, chow_tsai_1998, andrews_1999, gerthart_2014,  huang_liu_xu_2015, andrews_guan_ni_2016, guan_ni_2017}.
Equation (\ref{eq-thm-main}) can be generalized as follows.
\begin{equation}\label{eq-orlicz} dS_K=G(h_K)d\mu,\qquad \textnormal{on}\ \ \Sn,\end{equation}
where $\mu$ is a given Borel measure on $\Sn$.
The case where $G\equiv 1$ is the classical Minkowski problem. Minkowski's Existence and Uniqueness Theorem states that there exists a solution $K$ to the previous equation if and only if the barycentre of $\mu$ is the origin and $\mu$ is not concentrated on any great sub-sphere of $\Sn$ and, furthermore, any solution $K$ is unique up to translation. The case where $G(\theta)=\theta^{1-p}$, $p\in\R$ is the $L^p$ Minkowski problem introduced by Lutwak \cite{lutwak_1993}. While, as of the existence and uniqueness, the problem has been settled for $p\geq 1$ \cite{lutwak_1993, chou_wang_2006}, several problems remain open for $p\leq 1$ (see e.g. \cite{boroczky_lutwak_yang_zhang_2013, zhu_2014, bianchi_2019} for related results and open problems and \cite{huang_lyz_2016} for an important generalization). In particular, the question of uniqueness in the case $p=0$ and $\mu$, $h_K$ being even is now considered to be a major open problem in Convex Geometry (see for instance  \cite{stancu_2002, stancu_2003, boroczky_lutwak_yang_zhang_2012, saroglou_2015, Kolesnikov_milman_2020+}).

If $G$ is considered to be a general continuous function, (\ref{eq-orlicz}) was introduced in \cite{haberl_lyz_2010} (under some assumptions on $G$). While important results concerning existence have been obtained \cite{haberl_lyz_2010, jian_lu_2019,  }, very few facts are known about uniqueness. In view of (\ref{eq-curvature}), (\ref{eq-thm-main}) should be viewed as the constant curvature case in the Orlicz-Minkowski problem. 

Before we state our main results, we will need to agree on some notation. Given a positive integer $n$, set ${\cal A}(n)$ to be the positive cone, consisting of all continuous functions $G:(0,\infty)\to (0,\infty)$, such that for some (any) antiderivative $H$ of $G$, the function
$$(0,\infty)\ni\theta\mapsto\theta G(\theta)+nH(\theta)$$is strictly increasing. Notice that any differentiable function $G$ with $\theta G'(\theta)+(n+1)G(\theta)>0$, for all $\theta>0$, belongs to ${\cal A}(n)$.
 
\begin{theorem}\label{thm-main}
Let $K$ be a convex body in $\Rn$ that contains the origin in its interior and solves (\ref{eq-thm-main}) for some $G\in{\cal A}(n)$.
Then, $K$ is symmetric with respect to some straight line through the origin. In addition, $K$ is a Euclidean ball if at least one of the following hold
\begin{enumerate}[a)]
\item $b(K)=0$
\item $G$ is monotone.
\end{enumerate}In particular, if $G$ is strictly monotone or $b(K)=o$, then $K$ is a Euclidean ball centered at the origin.
\end{theorem}
Clearly, the function $\theta^p$ is contained in the class ${\cal A}(n)$, for $p>-n-1$, and any continuous strictly increasing function $G:(0,\infty)\to(0,\infty)$ is also contained in ${\cal A}(n)$. Hence, one immediately recovers Theorem \ref{thm-old-BCD} and Theorem \ref{thm-old-Simon}, in the case $k=n-1$, from Theorem \ref{thm-main}. Let us demonstrate that none of the assumptions of Theorem \ref{thm-main} can be removed. Indeed, if $G(\theta)=\theta^{-n-1}$, then (as mentioned previously) solutions of (\ref{eq-thm-main}) do not need to be axially symmetric even if $b(K)=o$, thus the assumption of $G$ being a member of ${\cal A}(n)$ cannot be omitted. The necessity of the monotonicity of $G$ follows from the next theorem.
\begin{theorem}\label{thm-counterexample}
Let $a\in\Rn$. There exists a centrally symmetric, non-spherical, strictly convex body $K$ in $\Rn$ with $C^\infty$ boundary and a function $G\in {\cal A}(n)$, such that $K+a$ contains the origin in its interior and satisfies
$$f_K(v)=f_{K+a}(v)=G(h_{K+a}(v)),\qquad\textnormal{for all }v\in\Sn.$$In fact, $K$ can be taken to be arbitrarily close to a Euclidean ball and $G$ can be taken to be arbitrarily close to a constant.\end{theorem}
We should remark that the method used to prove Theorem \ref{thm-main} is purely geometric and, therefore, quite different than the method used in \cite{brendle_choi_daskalopoulos_2017}. More specifically,  we employ a quick argument (although some preparation is needed; see Sections 2, 3, 4) based on Steiner-symmetrization (ultimately related to the Blaschke-Santal\'{o} inequality), to show that if $G\in {\cal A}(n)$, $K$ is a solution of (\ref{eq-thm-main}) and $H$ is a hyperplane through the origin that splits $K$ into two sets of equal volume, then $K$ has to be symmetric with respect to $H$, while if $b(K)=o$ then $K$ is a Euclidean ball. Based on this, we show in Section 5 that $K$ is always axially symmetric. Thus, the remaining part of the theorem essentially reduces to a 2-dimensional problem (1-dimensional actually if one attempts to solve the associated ODE). The result will follow by a careful modification of a solution $K$ (Sections 6 and 7), so that the resulting body is non-spherical, centrally symmetric and (approximately) solves (\ref{eq-thm-main}). Theorem \ref{thm-counterexample} will be proved in Section 8.


\section{Preliminaries}
\hspace*{1.5em}In this section, we fix some notation and state some basic facts about convex bodies that are necessary for our purposes. As general references, we state the books of Schneider \cite{schneider_2013} and Gardner \cite{gardner_2006}. We denote the origin by $o$ and the standard (Euclidean) unit ball of $\Rn$ by $B_2^n$. The closure, the interior, the boundary and the volume (i.e. Lebesgue measure) of a set $A$ will be denoted by $\cl A$ $\inter A$, $\bd A$ and $V(A)$ respectively. The orthogonal projection of a set or a vector onto a subspace $H$ will be denoted by $\cdot|H$.

A convex body $K$ is said to be \emph{regular} if the supporting hyperplane at each boundary point of $K$ is unique. It turns out that $K$ is regular if and only if its boundary is of class $C^1$. Furthermore, $K$ is of class $C_+^2$ if its boundary is of class $C^2$ (then, $h_K$ is also of class $C^2$) and the quantity $\textnormal{det}\left(h_{K,ij}(v)+\delta_{ij}h_K(v)\right)$ is strictly positive everywhere on $\Sn$.

The \emph{Gauss map} $$\eta_K:\bd K\to \Sn$$ takes every boundary point $x$ of the convex body $K$ to the (unique) outer unit normal vector that supports $K$ at $x$. If $K$ is strictly convex, then $\eta_K$ is invertible and $h_K$ is $C^1$. If, additionally, $K$ happens to be regular, then $\eta_K$ and $\eta_K^{-1}$ are homeomorphisms and $\eta_K^{-1}$ is given by

\begin{equation}\label{eq-gauss-map}
\eta_K^{-1}(v)=\nabla h_K(v),\qquad\textnormal{for all }v\in\Sn,
\end{equation}where $\nabla h_K$ is the usual gradient of $h_K$ in $\Rn$.
We also note that the surface area measure of any strictly convex body $K$ is absolutely continuous with respect to $\Ha.$ 

Let $L$, $M$ be convex bodies in $\Rn$. The \emph{first Minkowski mixed volume} $V(L,M)$ is defined by
$$V(L,M)=\frac{1}{n}\frac{d}{dt}V(tL+M)\Big|_{t=0^+}.$$
Here, $A+B:=\{x+y:x\in A,\ y\in B\}$ is the Minkowski sum of the sets $A$ and $B$. 
The following formula is well known
$$V(L,M)=\frac{1}{n}\int_{\Sn}h_{L}(v)dS_M(v).$$
A basic inequality concerning mixed volumes is Minkowski's first inequality, which reads as follows:
\begin{equation}\label{eq-Minkowski}V(L,M)\geq V(L)^{1/n}V(M)^{(n-1)/n}.\end{equation}

The \emph{polar body} $L^\circ$ of $L$ is defined to be the convex set  
$$L^\circ=\{x\in\Rn: \langle x,y\rangle \leq 1,\textnormal{ for all }y\in L\}.$$
The \emph{Santal\'{o} point} $s(L)$ of $L$ is defined to be the unique point $z\in \Rn$, such that $$V((L-z)^\circ)=\min\{V((L-x)^\circ):x\in\Rn\}.$$
Set, also, $L^*:=(L-s(L))^\circ$.
It is well known that $s(L)=o$ if and only if $b(L^\circ)=o$. Moreover, $L-s(L)$ contains the origin in its interior, while if $L$ contains the origin in its interior, then $L^\circ$ is also a convex body that contains the origin in its interior. In addition, it holds $(L^\circ)^\circ=L$.

The \emph{radial function} $\rho_L:\Sn\to \R$ is given by
$$\rho_L(v)=\sup\{\lambda\geq 0: \lambda v\in L\}.$$

If $o\in\inter L$, then the radial function of $L$ and the support function of the polar body $L^\circ$ are related by
\begin{equation}\label{eq-radial-function}
\rho_L(v)=\frac{1}{h_{L^\circ}(v)},\qquad \textnormal{for all v}\in\Sn.
\end{equation}

Throughout this paper, $e$ will be a fixed unit vector, unless stated otherwise. Let $J$ be a subinterval of the real line. A \emph{shadow system} along the direction $e$ is a family of convex bodies $\{L(t)\}_{t\in J}$ of the form
$$L(t)=\textnormal{conv}\{x+\beta(x)e:x\in L\},\qquad t\in J,$$for some bounded function $\beta:L\to\R$. 
Shadow systems where introduced in \cite{rogers_shephard_1958}. It is known \cite{rogers_shephard_1958,shephard_1964} (see also e.g. \cite{campi_gronchi_2006_a}, \cite{campi_gronchi_2006}) that several functionals (such as quermassintegrals) 
on convex bodies are convex with respect to the parameter $t$. We are going to use in particular that $h_{L(t)}(v)$ is convex in $t$, for any $v\in\Sn$. A simple consequence of this fact is the following lemma. 
\begin{lemma}\label{lemma-differentiability-1}
Let $L$ be a convex body in $\Rn$, $\{L(t)\}_{t\in[-1,1]}$ be a shadow system along direction $e$ and let $t_0\in[-1,1]$. Then, the function
$$(t,v)\mapsto \frac{h_{L(t)}(v)-h_{L(t_0)}(v)}{t-t_0}$$
is uniformly bounded on $([-1,1]\setminus\{t_0\})\times\Sn$.
\end{lemma}
\begin{proof}
One can (trivially) extend $\{L(t)\}_{t\in[-1,1]}$ to a shadow system $\{L(t)\}_{t\in\R}$. Since, for each $v\in\Sn$, $h_{L(t)}(v)$ is convex in $t$, we can write
\begin{eqnarray*}
-\infty&<&-\frac{\|h_{L(-2)}\|_{L^\infty(\Sn)}+\|h_{L(t_0)}\|_{L^\infty(\Sn)}}{3}\leq \frac{h_{L(-2)}(v)-h_{L(t_0)}(v)}{-2-t_0}\\
&\leq &\frac{h_{L(t)}(v)-h_{L(t_0)}(v)}{t-t_0}\leq \frac{h_{L(2)}(v)-h_{L(t_0)}(v)}{2-t_0}\leq \|h_{L(2)}\|_{L^\infty(\Sn)}+\|h_{L(t_0)}\|_{L^\infty(\Sn)}<\infty,
\end{eqnarray*}
for all $(t,v)\in ([-1,1]\setminus\{t_0\})\times \Sn$.
\end{proof}
The \emph{Steiner-symmetral} $St_e(L)$ of the convex body $L$ with respect to the hyperplane $e^\perp:=\{v\in\Rn:\langle e,v\rangle=0\}$ is the (apparently convex body) set obtained by replacing, for each $\overline{x}\in e^\perp$, the intersection of $L$ with the line that passes through $\overline{x}$ and is parallel to $e$, with the line segment of the same length which is symmetric with respect to the hyperplane $e^\perp$ and passes through $\overline{x}$. 

A particular case of a shadow system can be constructed as follows: The convex body $L$ can be written as 
\begin{equation}\label{eq-form}
L=\{\overline{x}+ye:\overline{x}\in L|e^\perp,\ z(\overline{x})\leq y\leq w(\overline{x})\},
\end{equation}
where the functions $z,w:L|e^\perp\to\R$ are such that $z$ is convex, $w$ is concave and $z\leq w$. Define, then,
\begin{eqnarray*}L_t:&=&\{x-(1-t)u(x|e^\perp)e:x\in St_e(L)\}\\
&=&\{\overline{x}+ye:\overline{x}\in L|e^\perp,\ z(\overline{x})-(1-t)u(\overline{x})\leq y\leq w(\overline{x})-(1-t)u(\overline{x})\},\qquad t\in[-1,1],\end{eqnarray*}where $u:=(z+w)/2$.
One can check that $L_t$ is convex for all $t\in[-1,1]$ and, therefore, the family $\{L_t\}_{t\in[-1,1]}$ is indeed a shadow system along direction $e$. Moreover, $L_1=L$, $L_{-1}$ equals the reflection of $L$ with respect to the hyperplane $e^\perp$ and $L_0=St_e(L)$. Thus, the shadow system $\{L_t\}_{t\in[-1,1]}$ is a way to perform Steiner-symmetrization to $L$ in a continuous way.

We will need some (relatively recent) results concerning the polar volume of a shadow system $\{L(t)\}_{t\in J}$.

\begin{oldtheorem}\label{old-thm-mr1}(Meyer-Reizner \cite{meyer_reisner_2006}) The function $J\ni t\mapsto V((L(t))^*)^{-1}$ is convex. 

When we restrict our attention to the shadow system $\{L_t\}_{t\in[-1,1]}$ that corresponds to Steiner-symmetrization, Theorem \ref{old-thm-mr1} easily implies the following 
\begin{corollary}\label{cor-mr1} (Meyer-Pazor  \cite{meyer_pajor_1990}; Meyer-Reisner \cite{meyer_reisner_2006}) It holds $$V(L^*)^{-1}\geq V((L_0)^*)^{-1}.$$
\end{corollary}
It should be remarked that the Blaschke-Santal\'{o} inequality follows from Corollary \ref{cor-mr1} in a standard way.

If  $L$ is a centrally symmetric non-ellipsoidal convex body, it was shown in \cite{saint_raymont_1981} and \cite{meyer_pajor_1990} that there exists a direction, such that the Steiner-symmetrization along this direction strictly increases the volume of $L^\circ$. Therefore, in view of Theorem \ref{old-thm-mr1}, this statement implies the following.
\end{oldtheorem}
\begin{lemma}\label{lemma-polar-vol-not-ellipsoid}Let $L$ be a centrally symmetric convex body in $\Rn$ which is not an ellipsoid. Then, there exists an orthogonal map $O:\Rn\to\Rn$ such that
$$\frac{d}{dt}V(((OL)_t)^\circ)\Big|_{t=1^-}<0.$$
\end{lemma}
Denote by $(e^\perp)^+$ and $(e^\perp)^-$ the two closed half-spaces defined by the hyperplane $e^\perp$. The following fact (stated in a simplified form; it also follows from the proof of the main result in \cite{campi_gronchi_2006}) is also proved in \cite[Lemma 4]{meyer_reisner_2006}.
\begin{oldtheorem}\label{old-thm-mr2}\cite{meyer_reisner_2006}
The function $J\ni t\mapsto V(L(t)^\circ\cap (e^\perp)^\pm)^{-1}$ is convex. 
\end{oldtheorem}
As a consequence, we have the following lemma.
\begin{lemma}\label{lemma-polar}
Let $L$ be a convex body containing the origin in its interior. If either $b(L^\circ)=o$ or $V(L^\circ\cap (e^\perp)^+)=V(L^\circ\cap (e^\perp)^-)$ then $$V(L^\circ)\leq V((L_t)^\circ),$$for all $t\in[-1,1]$.
\end{lemma}
\begin{proof}
If $b(L^\circ)=o$, then due to the minimality of $V((L_t)^*)$, we have
\begin{eqnarray*}
V((L_t)^\circ)^{-1}\leq V((L_t)^*)^{-1}\leq tV((L_1)^*)^{-1}+(1-t)V((L_0)^*)^{-1}=tV((L_1)^\circ)^{-1}+(1-t)V((L_0)^*)^{-1},
\end{eqnarray*}
where we used Theorem \ref{old-thm-mr1}. Thus, using Corollary \ref{cor-mr1}, we obtain
$$V((L_t)^\circ)^{-1}-V((L_1)^\circ)^{-1}\leq (1-t)(V((L_0)^*)^{-1}-V((L_1)^\circ)^{-1})=(1-t)(V((L_0)^*)^{-1}-V((L_1)^*)^{-1})\leq 0.$$This proves our first assertion. To prove the second one, notice that if $$V(L^\circ\cap (e^\perp)^+)=V(L^\circ\cap (e^\perp)^-),$$ then
$$V((L_1)^\circ\cap (e^\perp)^+)=V((L_1)^\circ\cap (e^\perp)^-)=V((L_{-1})^\circ\cap (e^\perp)^+)=V((L_{-1})^\circ\cap (e^\perp)^-).$$
Since by Theorem \ref{old-thm-mr2}, the function  $[-1,1]\ni t\mapsto V((L_t)^\circ\cap (e^\perp)^\pm)^{-1}$ is convex, it holds
\begin{eqnarray*}V((L_0)^\circ\cap (e^\perp)^\pm)^{-1}&\leq& \frac{1}{2}\left(V((L_1)^\circ\cap (e^\perp)^\pm)^{-1}+V((L_{-1})^\circ\cap (e^\perp)^\pm)^{-1}\right)\\
&=&V((L_1)^\circ\cap (e^\perp)^\pm)^{-1}=V((L_{-1})^\circ\cap (e^\perp)^\pm)^{-1}.\end{eqnarray*}
This shows that the functions $t\mapsto V((L_{t})^\circ\cap (e^\perp)^\pm)^{-1}$ are non-increasing on $[-1,0]$ and non-decreasing on $[0,1]$, hence the function $t\mapsto V((L_t)^\circ)=V((L_{t})^\circ\cap (e^\perp)^+)+V((L_{t})^\circ\cap (e^\perp)^-)$ is non-decreasing on $[-1,0]$ and non-increasing on $[0,1]$. This proves our second claim.
\end{proof}
\section{Some facts concerning the class ${\cal A}(n)$}
\hspace*{1.5em}Below, we collect some general statements on functions from ${\cal A}(n)$ that (although their proofs are simple) will be crucial for the proof of Theorem \ref{thm-main}. 
\begin{lemma}\label{lemma-facts-extension}
Let $0<a<b$ be positive numbers and $G:[a,b]\to (0,\infty)$ be a continuous function, such that if $H$ is an antiderivative of $G$, then the function
$$[a,b]\ni\theta\mapsto\theta G(\theta)+nH(\theta)$$
is strictly increasing. Then, $G$ can be extended to an ${\cal A}(n)$ function $\overline{G}:(0,\infty)\to(0,\infty)$, such that there exists a continuous strictly increasing function $F:[0,\infty)\to\R$, with $F(0)=0$, satisfying
\begin{equation}\label{eq-lemma-facts}\overline{G}(\theta)=\left(\int_0^1r^{n-1}F(r\theta)dr\right)',\qquad\textnormal{for all }\theta\in(0,\infty).\end{equation}
\end{lemma}
\begin{proof}
Set $$\overline{H}(\theta):=\begin{cases}
\theta G(a), & 0\leq\theta<a\\
H(\theta)-H(a)+aG(a), & a\leq\theta\leq b\\
(\theta-b)G(b)+H(b)-H(a)+aG(a), & \theta > b
\end{cases}.$$Then, $\overline{H}$ is $C^1$ and strictly increasing. Set, also, $\overline{G}:=\overline{H}'$.
The assumption on $G$ implies that $\overline{G}|_{(0,\infty)}\in{\cal A}(n)$. Let $F(\theta):=\theta \overline{H}'(\theta)+n\overline{H}(\theta)=\theta \overline{G}(\theta)+n\overline{H}(\theta)$ and 
\begin{equation}\label{eq-lemma-facts-extension}I(\theta):=\int_0^1r^{n-1}F(r\theta)dr=\frac{\int_0^\theta s^{n-1}F(s)ds}{\theta^n},\qquad \theta\in (0,\infty).
\end{equation} Notice that $F(0)=\overline{H}(0)=\lim_{\theta\to 0^+}I(\theta)=0$. Differentiating (\ref{eq-lemma-facts-extension}), we find
$$I'(\theta)=\frac{F(\theta)}{\theta}-n\frac{I(\theta)}{\theta}$$
and, hence,
$$\theta I'(\theta)+nI(\theta)=\theta\overline{H}'(\theta)+n\overline{H}(\theta),\qquad\textnormal{for all }\theta\in (0,\infty).$$This easily implies that there exists a constant $c\in\R$, such that 
$$I(\theta)=\overline{H}(\theta)+c\theta^{-n},\qquad 
\theta\in (0,\infty).$$
Since $\overline{H}(0)=\lim_{\theta\to 0^+}I(\theta)=0$, we conclude that $I=\overline{H}$, hence (\ref{eq-lemma-facts}) follows from (\ref{eq-lemma-facts-extension}).
\end{proof}
\begin{remark}\label{remark-we-may-assume} When studying (\ref{thm-main}), for some function $G\in{\cal A}(n)$ and for some convex body $K$ containing the origin in its interior, we are only interested in the restriction of $G$ onto the interval $[\min_{v\in\Sn}h_K(v),\max_{v\in\Sn}h_K(v)]$. Thus, Lemma \ref{lemma-facts-extension} in particular shows that we may always assume that there exists a continuous strictly increasing function $F:[0,\infty)\to\R$, satisfying $F(0)=0$ and $G(\theta)=\left(\int_0^1r^{n-1}F(r\theta)dr\right)'$ , for all $\theta\in(0,\infty)$. 
\end{remark}
\begin{lemma}\label{lemma-facts-G_varepsilon}
Let $G\in {\cal A}(n)$ and let $0<c_1<c_2$ and $0<a_1<a_2$. If $G(c_1)\geq a_1$ and $G(c_2)\leq a_2$, then for some $\varepsilon_0>0$ and for every $\varepsilon\in(0,\varepsilon_0)$, there exists a function $G_\varepsilon\in {\cal A}(n)$, such that the family $\{G_\varepsilon\}_{\varepsilon\in(0,\varepsilon_0)}$ enjoys the following properties.
\begin{enumerate}[i)]
\item For any $\varepsilon\in(0,\varepsilon_0)$, it holds $G_\varepsilon(\theta)=G(\theta)$, for all $\theta\in[c_1+\varepsilon,c_2+\varepsilon]$, while $G_\varepsilon(c_i)=a_i$, $i=1,2$.
\item The function $G(\varepsilon,\theta):=G_\varepsilon(\theta)$ is bounded on $(0,\varepsilon_0)\times[c_1,c_2]$.
\end{enumerate}
\end{lemma}
\begin{proof}
Let $\varepsilon_0>0$ be any number, such that $c_1+\varepsilon_0<c_2-\varepsilon_0$ and let $\varepsilon\in(0,\varepsilon_0)$. If $G(c_1)=a_1$ (resp. $G(c_2)=a_2$) set $c_1':=c_1$ (resp. $c_2':=c_2$), while if $G(c_1)>a_1$ (resp. $G(c_2)<a_2$), fix $c_1'$ to be any real number in $(c_1,c_1+\varepsilon)$ (resp. fix $c_2'$ in $(c_2-\varepsilon,c_2)$), such that $G(c_1')\geq a_1$ (resp. $G(c_2')\leq a_2$). 
Define $G_\varepsilon:[c_1,c_2]\to(0,\infty)$ by
$$G_\varepsilon(\theta)=\begin{cases}
a_1+(G(c'_1)-a_1)\frac{\theta-c_1}{c_1'-c_1},& c_1\leq \theta\leq c'_1\\
G(\theta),& c'_1< \theta< c'_2\\
G(c'_2)+(a_2-G(c'_2))\frac{\theta-c'_2}{c_2-c'_2}, & c'_2\leq \theta\leq c_2
\end{cases}$$
and set $H_\varepsilon(\theta):=\int_0^\theta G_\varepsilon(r)dr$, $\theta>0$. Then, $G_\varepsilon$ is continuous and coincides with $G$  on $[c_1+\varepsilon,c_2-\varepsilon]$. Moreover, $\theta G_\varepsilon(\theta)+nH_\varepsilon(\theta)$ is strictly increasing on $[c_1,c'_1]$ and $[c'_2,c_2]$. Furthermore, $H_\varepsilon|_{[c'_1,c'_2]}$ is an antiderivative of $G_\varepsilon|_{[c'_1,c'_2]}=G|_{[c'_1.c'_2]}$, so the function $\theta G_\varepsilon(\theta)+nH_\varepsilon(\theta)$ is also strictly increasing on $[c'_1,c'_2]$. 
The continuity of $\theta G_\varepsilon(\theta)+nH_\varepsilon(\theta)$ shows that $\theta G_\varepsilon(\theta)+nH_\varepsilon(\theta)$ is strictly increasing on $[c_1,c_2]$. Then, by Lemma \ref{lemma-facts-extension}, one can extend $G_\varepsilon$ to a function from ${\cal A}(n)$. Finally, notice that $G_\varepsilon(\theta)\le \max\{a_2,\max_{\theta\in[c_1,c_2]}G(\theta)\}$, for all $\theta\in[c_1,c_2]$.
\end{proof}

\section{Differentiability properties of Steiner-symmetrization}
We start this section with the following simple lemma.
\begin{lemma}\label{lemma-differentiability-2}
Let $L$ be a convex body in $\Rn$ that contains the origin in its interior. For $x\in \bd L$, set $N(L,x)$ to be the set of outer unit normal vectors to hyperplanes that support $L$ at $x$. Set, also,
$$m(L):=\inf\left\{\frac{|\langle v,x\rangle|}{|x|}:x\in \bd L_t,\ v\in N(L_t,x), \ t\in[-1,1]\right\}.$$
Then, it holds $m(L)>0$.
\end{lemma}
\begin{proof}
For $t\in[-1,1]$, one can check that $L_t$ contains the origin in its interior. Let $t\in [-1,1]$, $x\in \bd L_t$ and $v\in N(L_t,x)$. Then, $x$ cannot be orthogonal to $v$; otherwise, the supporting hyperplane $H$ of $L_t$, whose outer unit normal vector is $v$, would be parallel to the line segment $[o,x]$ and, therefore, $[o,x]$ would be contained in $H$. This would show that the origin is not contained in the interior of $L_t$, a contradiction. Consequently, $|\langle v,x\rangle|/|x|>0$. A continuity/compactness argument easily yields our claim. 
\end{proof}
As a consequence, we have the following.
\begin{proposition}\label{proposition-differentiability-1}
Let $L$ be a convex body in $\Rn$ that contains the origin in its interior and let $t_0\in[-1,1]$. Then, the function
$$([-1,1]\setminus\{t_0\})\times \Sn\ni (t,v)\mapsto \frac{\rho_{L_t}(v)-\rho_{L_{t_0}}(v)}{t-t_0}$$
is uniformly bounded.
\end{proposition}
\begin{proof}
Fix $v\in \Sn$. Let $t_1,t_2\in[-1,1]$ with $t_1\neq t_2$ and assume for instance that $\rho_{L_{t_1}}(v)>\rho_{L_{t_2}}(v)$. Set $x:=\rho_{L_{t_1}}(v)v$ and $y:=\rho_{L_{t_2}}(v)v$. Clearly, $x\in \bd L_{t_1}$ and $y\in\bd L_{t_2}$. Choose $v'\in N(L_{t_1},x)$. Then, we have
$$h_{L_{t_1}}(v')=\langle v',x\rangle=\frac{\langle v',x\rangle}{|x|}\rho_{L_{t_1}}(v),$$
while
$$h_{L_{t_2}}(v')\geq \langle v',y\rangle=\langle v',v\rangle\rho_{L_{t_2}}(v)=\frac{\langle v',x\rangle}{|x|}\rho_{L_{t_2}}(v).$$
Consequently, it holds
\begin{equation}\label{eq-proposition-differentiability-1}
\left|\frac{\rho_{L_{t_1}}(v)-\rho_{L_{t_2}}(v)}{t_1-t_2}\right|\leq \left(\frac{|\langle v',x\rangle|}{|x|}\right)^{-1}\left|\frac{h_{L_{t_1}}(v')-h_{L_{t_2}}(v')}{t_1-t_2}\right|.
\end{equation}
Our claim follows immediately from Lemma \ref{lemma-differentiability-1} and Lemma \ref{lemma-differentiability-2}, after setting $t_1=t$ and $t_2=t_0$ in (\ref{eq-proposition-differentiability-1}).
\end{proof}
In the next section, we will need that $\rho_{L_t}(v)$ is differentiable from the left at $t=1$. This would follow immediately if we knew that $\rho_{L_t}(v)$ posses some concavity property (for instance, it has been mentioned that $h_{L_t}$ is convex in $t$). Unfortunately, since we are not aware of any such property, we need to work a little bit harder in order to establish the desired differentiability.     
\begin{proposition}\label{prop-differentiability-2}
Let $L$ be a convex body in $\Rn$, containing the origin in its interior. Then, for all $v\in\Sn$ and for all $t_0\in(-1,1]$, $\rho_{L_t}(v)$ is differentiable from the left at $t=t_0$.
\end{proposition}
\begin{proof}
It is clearly enough to assume that $n=2$ (just take intersections of $L$ with 2-dimensional subspaces containing $e$). After a possible rotation of the coordinate system, $L_t$ can be written as 
$$L_t=\left\{(x,y)\in\R^2: x\in[a,b],\ z(x)-(1-t)\frac{w(x)+z(x)}{2}\le y\le w(x)-(1-t)\frac{w(x)+z(x)}{2}\right\},$$for some $a<0<b$ and some functions $w,\ z:[a,b]\to\R$, such that $z\leq  w$ and $z,-w$ are convex. Set $u:=(w_L+z_L)/2$ and fix $v=(v_1,v_2)\in\mathbb{S}^1$, $t_0\in[-1,1]$. For $t\in[-1,1]$, set also $x_t:=\rho_{L_t}(v)v_1$ and $y_t:=\rho_{L_t}(v)v_2$. If $v_1=0$, it is clear that $d\rho_{L_{t_0}}(v)/dt|_{t=t_0}=\pm u(0)$. We may, therefore, assume that $v_1>0$ (the case $v_1<0$ follows from the case $v_1>0$ by considering the reflections of $L$ and $v$ with respect to the $y$-axis). 
Notice, then, that $x_t\in(0,b]$ and that $(x_t,y_t)=\rho_{L_t}(v)$ is continuous in $t$. 

\emph{Claim 1.} Let $t_0\in(-1,1]$. There exists $\delta>0$, such that at precisely one of the following holds.
\begin{enumerate}[i)]
\item $y_t=w(x_t)-(1-t)u(x_t)$ and $x_t\neq x_{t_0}$, for all $t\in [t_0-\delta,t_0)$.
\item $y_t=z(x_t)-(1-t)u(x_t)$ and $x_t\neq x_{t_0}$, for all $t\in [t_0-\delta,t_0)$.
\item $x_t=x_{t_0}$, for all $t\in [t_0-\delta,t_0)$.
\end{enumerate}

To prove Claim 1, we will first show that either $x_t\neq x_{t_0}$ in a left neighbourhood of $t=t_0$ or $x_t=x_{t_0}$ in a left neighbourhood of $t=t_0$. If $z(x_{t_0})-(1-t_0)u(x_{t_0})<y_{t_0}<w(x_{t_0})-(1-t_0)u(x_{t_0})$ or $u(x_{t_0})=0$, then it is clear that $x_{t_0}=b$ and the point $(b,y_{t_0})$ remains a boundary point of $L_t$, for all $t$ in a neighbourhood of $t_0$. In particular, we fall in case (iii). We may, therefore, assume that $u(x_{t_0})\neq 0$ and either it holds $y_{t_0}=w(x_{t_0})-(1-t_0)u(x_{t_0})$  or it holds $y_{t_0}=z(x_{t_0})-(1-t_0)u(x_{t_0})$. For $t\in[-1,1]$, consider the line segment $I_t$ (which degenerates to a point if $x_{t_0}=b$ and $z(b)=w(b)$), defined as the intersection of $L_t$ with the line which is parallel to the $x_2$-axis and passes through the point $(x_{t_0},0)$. Then, $(x_{t_0},y_{t_0})$ is an end-point of $I_{t_0}$ and $I_t=I_{t_0}+(t-t_0)u(x_{t_0})(0,1)$, $t\in[-1,1]$. Thus, (depending on the sign of $u(x_{t_0})$) either $I_t$ does not contain $(x_{t_0},y_{t_0})$, for any $t\in[-1,t_0)$ (this is always true if $I_{t_0}$ is a singleton) or $(x_{t_0},y_{t_0})$ is an interior point of $I_t$, for $t$ lying in some interval of the form $[t_0-\delta,t_0)$. In the first case, we clearly have $x_t<x_{t_0}$, for all $t\in[-1,t_0)$. Let us consider the second case. If $x_{t_0}<b$, then $x_t$ is an interior point of $L_t$, for all $t\in[t_0-\delta,t_0)$. In particular, $x_t>x_{t_0}$, for all $t\in[t_0-\delta,t_0)$. If $x_{t_0}=b$, then the point $(b,y_{t_0})$ remains a boundary point of $L_t$ for $t$ close to $t_0$. Consequently, $x_t=x_{t_0}=b$, for all $t\in[t_0-\delta,t_0)$. 

To prove the remaining assertions of our claim, observe that since $(x_t,y_t)$ is continuous in $t$ and since $z(x)-(1-t)u(x)<w(x)-(1-t)u(x)$, for all $x\in (a,b)$, it follows that if $J$ is a subinterval of $[-1,1]$, such that $x_t\in(a,b)$, then either it holds $y_t=w(x_t)-(1-t)u(x_t)$, for all $t\in J$ or it holds $y_t=z(x_t)-(1-t)u(x_t)$, for all $t\in J$. Notice that if $x_{t_0}<b$, then we may assume (by choosing a smaller $\delta$ if necessary) that $x_t<b$, for all $t\in[t_0-\delta,t_0)$. Thus, if $x_t\neq x_{t_0}$ in $[t_0-\delta,t_0)$, one can take $J=[t_0-\delta,t_0)$, completing the proof of Claim 1. $\square$

Next, fix $t_0\in(-1,1]$. We will show that if $x_{t_0}<b$, then $\rho_{L_t}(v)$ is differentiable from the left at $t=t_0$. Since,  
\begin{equation}\label{eq-prop-differentiability-2-1}
\rho_{L_t}(v)=\sqrt{1+\left(\frac{v_2}{v_1}\right)^2}x_t, \qquad t\in[-1,1],
\end{equation}
it follows that case (iii) in Claim 1 can be excluded. By considering the reflections of $L$ and $v$ with respect to the $x_1$-axis (if necessary), we may assume that case (i) occurs.  For $t\in[-1,1]$ and $x\in(0,b)$, set 
$$F(t,x):=\frac{w(x)-(1-t)u(x)}{x}$$
and notice that $F(t,x_t)= v_2/v_1$, for all $t\in[t_0-\delta,t_0)$.

First assume that $x_t<x_{t_0}$ in $[t_0-\delta,t_0)$.  Remark that the concavity of $w$ implies that the left derivatives $w'_-(x)$ and $u'_-(x)$ exist for all $x\in(a,b)$. Using this, we easily arrive at
\begin{equation}\label{eq-prop-differentiablity-2-2}
\lim_{{(t,x)\to(t_0,x_{t_0})}\atop{t<t_0,\ x<x_{t_0}}}\frac{F(t,x)-F(t,x_{t_0})}{x-x_{t_0}}=\frac{x_1w'_-(x_{t_0})-w(x_{t_0})}{x_{t_0}^2}.
\end{equation}
Since
$$0=\frac{F(t,x_t)-F(t_0,x_{t_0})}{t-t_0}=\frac{F(t,x_t)-F(t_0,x_{t_0})}{x_t-x_{t_0}}\frac{x_t-x_{t_0}}{t-t_0}+\frac{u(x_{t_0})}{x_{t_0}},$$ for $t\in[t_0-\delta,t_0)$ and since
(\ref{eq-prop-differentiablity-2-2}) holds, we conclude that 
$$\lim_{t\to t_0^-}\frac{x_t-x_{t_0}}{t-t_0}=-\frac{u(x_{t_0})}{x_{t_0}}\frac{x_{t_0}^2}{x_{t_0}w'_-(x_{t_0})-w(x_{t_0})},$$as long as $x_{t_0}w'_-(x_{t_0})-w(x_{t_0})\neq 0$. But if $x_{t_0}w'_-(x_{t_0})-w(x_{t_0})$ was equal to zero, then (since $x_{t_0}u(x_{t_0})\neq 0$), the ratio $(x_t-x_{t_0})/(t-t_0)$ would be unbounded. This, by (\ref{eq-prop-differentiability-2-1}) and Proposition \ref{proposition-differentiability-1} would lead us to a contradiction and, therefore (again by (\ref{eq-prop-differentiability-2-1})) 
$(d/dt)\rho_{L_t}(v)|_{t=t_0^-}$ exists and it holds
\begin{equation}\label{eq-prop-differentiablility-2-3}
\frac{d}{dt}\rho_{L_t}(v)\Big|_{t=t_0^-}=-\sqrt{1+\left(\frac{v_2}{v_1}\right)^2}\frac{x_{t_0}u(x_{t_0})}{x_{t_0}w'_-(x_{t_0})-w(x_{t_0})}.
\end{equation} 

The case where $x_t>x_{t_0}$ in a left neighbourhood of $t=t_0$ can be treated similarly; one has to replace $w'_-(x)$ and $u'_-(x)$ by $w'_+(x)$ and $u'_+(x)$ in the previous argument.

Finally, we need to show that $(d/dt)\rho_{L_t}(v)$ is differentiable from the left at $t=t_0$, if $x_{t_0}=b$. As before, we may assume that there exists $\delta>0$, such that  $x_t<b$ and $y_t=w(x_t)-(1-t)u(x_t)$, for all $t\in[t_0-\delta,t_0)$. Then, (\ref{eq-prop-differentiablility-2-3}) together with the concavity of $w$ shows that $(d/dt)\rho_{L_t}(v)|_{t=s^-}$ exists for any $s\in(t_0-\delta,t_0)$ and the limit 
$$\lim_{s\to t_0^-} \frac{d}{dt}\rho_{L_t}(v)\Big|_{t=s^-}$$also exists and is finite (notice that this is true even if $w'_-(b)=-\infty$). This easily shows that the left derivative $(d/dt)\rho_{L_t}(v)|_{t=t_0^-}$ exists.
\end{proof}

Taking (\ref{eq-radial-function}) into account, an immediate consequence of Proposition \ref{proposition-differentiability-1} and Proposition \ref{prop-differentiability-2} is the following.

\begin{corollary}\label{prop-diff}
Let $L$ be a convex body in $\Rn$ that contains the origin in its interior. Then, the ratio
$\frac{h_{(L_t)^\circ}(v)-h_{L^\circ}(v)}{t-1}$ is uniformly bounded in $[-1,1)\times \Sn$ and the left derivative $\frac{d}{dt}h_{(L_t)^\circ}(v)\Big|_{t=1^-}$ exists for all $v\in \Sn$.
\end{corollary}

The next result is the infinitesimal version (its proof follows more or less the same lines) of the well known fact that the volume of the intersection of a convex body with the unit ball $B_2^n$ increases under Steiner-symmetrization. 
\begin{proposition}\label{lemma-prop-st-ym}
Let $L$ be a convex body in $\Rn$ and $a>0$. Then, the function $[-1,1]\ni t\mapsto V(B_2^n\cap aL_t)$ is differentiable from the left at $t=1$. Moreover, it holds \begin{equation}\label{eq-lemma-prop-st-sym}\frac{d}{dt}V(B_2^n\cap aL_t)\Big|_{t=1^-}\leq 0.\end{equation}Finally, if $L$ is not symmetric with respect to the hyperplane $e^\perp$, then there exists an open interval $J\subseteq (0,\infty)$, such that for any $a\in J$, inequality (\ref{eq-lemma-prop-st-sym}) is strict. 
\end{proposition}
\begin{proof}
To prove the differentiability property and (\ref{eq-lemma-prop-st-sym}), it suffices to assume that $a=1$. Write $L$ in the form (\ref{eq-form}), for some concave functions $-z,w:L|e^\perp\to\R$ and set $u:=(z+w)/2$. For $x\in e^\perp$, $t\in[-1,1]$, set $\Phi(t,x):=V_1(B_2^n\cap L_t\cap (x+\R x))$ to be the length of the intersection of $B_2^n$, the convex body $L_t$ and the line which passes through $x$ and is parallel to $e$. Since $\Phi(t,x)=V_1(B_2^n\cap (x+\R e)\cap (L_0\cap (x+\R x)+tu(x)e))$, it is well known (and easily verified) that $\Phi(t,x)$ is log-concave in $t$. In fact, if $B_2^n\cap (x+\R e)=[x-ae,x+ae]$, for some $a>0$, then
$$\Phi(t,x)=\min\{w(x)-(1-t)u(x),a\}-\max\{z(x)-(1-t)u(x),-a\},$$
while if $B_2^n\cap (x+\R e)=\emptyset$, then $\Phi(t,x)\equiv 0$. In particular, $t\mapsto\Phi(t,x)$ is differentiable from the left at $t=1$ and is also Lipschitz with Lipschitz constant bounded by $2\sup_{x\in L|e^\perp}|u(x)|<\infty$. Then, Fubini's Theorem and the Bounded Convergence Theorem yield
\begin{equation}\label{eq-l-dec-st-s-0}\lim_{t\to 1^-}\frac{V(B_2^n\cap L)-V(B_2^n\cap L_t)}{1-t}=\int_{e^\perp}\lim_{t\to 1^-}\frac{\Phi(1,x)-\Phi(t,x)}{1-t}dt=\int_{e^\perp}\frac{d}{dt}\Phi(t,x)\Big|_{t=1^-}dx.
\end{equation}
This proves our first assertion.

Notice that $B_2^n\cap (x+\R e)$ and $L_0\cap (x+\R e)$ are both symmetric line segments (possibly degenerate or empty) with respect to the hyperplane $e^\perp$. Therefore, $\Phi(t,x)$ is even in $t$. This, together with the log-concavity property mentioned above shows that the function $[-1,1]\ni t\mapsto \Phi(t,x)$ is non-decreasing on $[-1,0]$ and non-increasing on $[0,1]$, thus 
\begin{equation}\label{eq-l-dec-st-s}
\Phi(1,x)=\Phi(-1,x)\leq \Phi(t,x),\qquad \textnormal{for all }t\in[-1,1]
\end{equation}
and 
\begin{equation}\label{eq-l-dec-st-s-2}
\frac{d}{dt}\Big|_{t=1^-}\Phi(t,x)\leq 0.
\end{equation}
Equations (\ref{eq-l-dec-st-s-0}) and (\ref{eq-l-dec-st-s-2}) immediately give (\ref{eq-lemma-prop-st-sym}).

Next, assume that the line segments $L\cap(x+\R e)$, $B_2^n\cap (x+\R e)$ satisfy the following:
\begin{enumerate}[a)]
\item $L\cap (x+\R e)$ is not symmetric with respect to the hyperplane $e^\perp$.
\item $L\cap (x+\R e)\not\subseteq B_2^n\cap (x+\R e)$ and $B_2^n\cap (x+\R e)\not\subseteq L\cap (x+\R e)$.
\item $V_1(B_2^n\cap L\cap (x+\R e)>0$.
\end{enumerate}
Then, assumptions (a)-(c) together with (\ref{eq-l-dec-st-s}) and the log-concavity of $t\mapsto \Phi(t,x)$ ensure that the function $\Phi(t,x)$ is strictly decreasing in $t$ on $[0,1]$. Moreover, using again the log-concavity property together with assumption (c), we easily conclude that $d\Phi(t,x)/dt|_{t=1^-}$ is strictly negative. This, together with (\ref{eq-l-dec-st-s-0}), (\ref{eq-l-dec-st-s-2}) and a continuity argument, shows that whenever there exists a point $x\in e^\perp$ satisfying assumptions (a)-(c), inequality (\ref{eq-lemma-prop-st-sym}) is strict with $a=1$. 

To finish with our proof, simply observe that if $L$ is not symmetric with respect to $e^\perp$, then there exists $a>0$ and a point $x\in e^\perp$ that satisfies assumptions (a)-(c) when $L$ is replaced with $aL$. Another continuity argument shows that the same is true in a whole neighbourhood $J$ of $a$.
\end{proof}

\section{Axial symmetry of solutions}
\hspace*{1.5em}The main goal of this section is to establish the most important part of Theorem \ref{thm-main}. More specifically, we prove the following.
\begin{proposition}\label{lemma-key}
Let $K$ be a convex body in $\Rn$ that contains the origin in its interior and solves (\ref{eq-thm-main}) for some $G\in{\cal A}(n)$. Then, $K$ is symmetric with respect to some straight line through the origin. Moreover, if $b(K)=o$, then $K$ is a Euclidean ball centered at the origin.
\end{proposition}

For a convex body $L$ that contains the origin in its interior, set $L^t:=((L^\circ)_t)^\circ$, $t\in[-1,1]$. For $v\in\Sn$, set also $h'_{L}(v)$ to be the left derivative of $h_{L^t}(v)$ at $t=1$ (which exists always thanks to Corollary \ref{prop-diff}).
The following two lemmas will be of great importance towards the proof of Proposition \ref{lemma-key}.
\begin{lemma}\label{lemma-eq-l-st-s-der}Let $L$ be a convex body that contains the origin in its interior and is not symmetric with respect to the hyperplane $e^\perp$ and let $G\in{\cal A}(n)$. Then,
$$\int_{\mathbb{S}^{n-1}}G(h_{L}(v))h'_L(v)d\Ha(v)>0.$$
\end{lemma}
\begin{proof} Let $F:[0,\infty)\to\R$ be the function associated with $G$ in Remark \ref{remark-we-may-assume} and set $H(\theta):=\int_0^1r^{n-1}F(r\theta)dr$. Set, also, $I(t):=\int_{\Sn}H(h_{L^t}(v))d\Ha(v)$. Corollary \ref{prop-diff} shows that
$$I'_-(1)=\int_{\Sn}G(h_L(v))h'_L(v)d\Ha(v).$$
Using polar coordinates, we obtain
\begin{eqnarray*}
I(t)=\int_{\Sn}H(h_{L^t}(v))d\Ha(v)&=&\int_{\Sn}\int_0^1r^{n-1}F(rh_{L^t}(v))drd\Ha(v)\\
&=&\int_{\Sn}\int_0^1 r^{n-1}F(h_{L^t}(rv))drd\Ha(v)=\int_{B_2^n}F(h_{L^t}(x))dx.\\
\end{eqnarray*}
Thus, using the layer cake formula, the continuity and the strict monotonicity of $F$, we can write
\begin{eqnarray*}
I(t)&=&\int_0^\infty V(B_2^n\cap \{F(h_{L^t})\geq s\})ds\\
&=&\int_0^{\sup F} V(B_2^n\cap \{F(h_{L^t})\geq s\})ds\\
&=&\int_0^{\sup F} [V(B_2^n)- V(B_2^n\cap \{F(h_{L^t})< s\})]ds\\
&=&\int_0^{\sup F} [V(B_2^n)- V(B_2^n\cap \{h_{L^t}< F^{-1}(s)\})]ds\\
&=&\int_0^{\sup F} [V(B_2^n)- V(B_2^n\cap F^{-1}(s)(L^\circ)_t)]ds\\
\end{eqnarray*}
Therefore, from Proposition \ref{lemma-prop-st-ym} and Fatou's Lemma, we easily get 
$$I'_-(1)=\liminf_{t\to 1^-}\frac{I(t)-I(1)}{t-1}\geq \int_0^{\sup F} \frac{d}{dt}[-V(B_2^n\cap F^{-1}(s)(L^\circ)_t)]\Big|_{t=1^-}ds>0,$$where we additionally used the fact that since $L$ is not symmetric with respect to the hyperplane $e^\perp$, $L^\circ$ also cannot be symmetric with respect to $e^\perp$.
\end{proof}

\begin{lemma}\label{lemma-before-key}
Let $L$ be a convex body in $\Rn$ that contains the origin in its interior.
\begin{enumerate}[i)]
\item If $b(L)=o$ or $V(L\cap (e^\perp)^+)=V(L\cap (e^\perp)^-)$, then it holds
\begin{equation*}\label{eq-lemma-before-key-main-1}
\int_{\Sn}h'_L(v)dS_L(v)\leq 0.
\end{equation*}
\item If $L$ is centrally symmetric and not an ellipsoid, then there exists an orthogonal map $O$, such that 
\begin{equation*}\label{eq-before-key-main-2}
\int_{\Sn}h'_{OL}(v)dS_{OL}(v)<0.\end{equation*}
\end{enumerate}
\end{lemma}
\begin{proof}
To prove (i), let $t\in[-1,1)$. Using Minkowski's first inequality (\ref{eq-Minkowski}), we obtain
\begin{equation}\label{eq-before-key}
\frac{V(L^t,L)-V(L)}{t-1}\leq \frac{V(L^t)^{1/n}V(L)^{(n-1)/n}-V(L)}{t-1}=V(L)^{(n-1)/n}\frac{V(L^t)^{1/n}-V(L)^{1/n}}{t-1}.
\end{equation}
Since by Lemma \ref{lemma-polar}, it holds $V(L^t)\geq V(L)$, we immediately see that
$$\liminf_{t\to 1^-}\frac{V(L^t,L)-V(L)}{t-1}\leq 0.$$
On the other hand, using Fatou's Lemma and Corollary \ref{prop-diff}, we get
\begin{eqnarray}\liminf_{t\to 1^-}\frac{V(L^t,L)-V(L)}{t-1}&=&\liminf_{t\to 1^-}\frac{1}{n}\int_{\Sn}\frac{h_{L^t}(v)-h_L(v)}{t-1}dS_L(v)\nonumber\\
&\geq& \frac{1}{n}\int_{\Sn}\liminf_{t\to 1^-}\frac{h_{L^t}(v)-h_L(v)}{t-1}dS_L(v)=\frac{1}{n}\int_{\Sn}h'_L(v)dS_L(v)\label{eq-before-key-2}.\end{eqnarray}Thus,
(i) holds.

Next assume that $L$ and, therefore, $L^\circ$ is centrally symmetric but not an ellipsoid. Using Lemma \ref{lemma-polar-vol-not-ellipsoid}, we conclude that there exists an orthogonal map $O$ such that 
$$\lim_{t\to 1^-}V(OL)^{(n-1)/n}\frac{V((OL)^t)^{1/n}-V(OL)^{1/n}}{t-1}<0.$$ 
Using $OL$ in the place of $L$ in (\ref{eq-before-key}) and (\ref{eq-before-key-2}), we conclude the validity of (ii).
\end{proof}

The second part of Lemma \ref{lemma-before-key} is going to be used in Section 7.\\
\\
\noindent Proof of Proposition \ref{lemma-key}.\\
First, notice that Lemma \ref{lemma-eq-l-st-s-der} and Lemma \ref{lemma-before-key} (i) imply the following: If $b(K)=o$ or $V(K\cap (e^\perp)^+)=V(K\cap (e^\perp)^-)$, then $K$ is symmetric with respect to the hyperplane $e^\perp$. 
In fact, if $b(K)=o$, the previous statement shows that $K$ is symmetric with respect to the hyperplane $e^\perp$, for every $e\in\Sn$. This clearly shows that $K$ is a Euclidean ball centered at the origin.

\emph{Fact 1. Let $E$ be a subspace of $\Rn$ of dimension at most $n-2$. Then, there exists a hyperplane $H$, containing $E$, such that $K$ is symmetric with respect to $H$.}\\ To see this, let $V$ be any $(n-2)$-dimensional subspace of $\Rn$ that contains $E$. Then, as in \cite{lovasz_simonovits_1993}, the trivial continuity argument shows that there exists a hyperplane $H$, containing $V$, such that $V(K\cap H^+)=V(K\cap H^-)$. Thus, according to our previous discussion, $K$ has to be symmetric with respect to the hyperplane $H$.

\emph{Fact 2. For $1\le k\le n-1$, there exist $e_1,\dots,e_{k}$ mutually orthogonal unit vectors, such that $K$ is symmetric with respect to the hyperplane $e_i^\perp$, $i=1,\dots,k$.}\\
We will prove Fact 2 using induction in $k$. The case $k=1$ follows from Fact 1. Assume that we have found $e_1,\dots,e_{k-1}$ as above for some $2\le k\le n-1$ and set $H:=e_1^\perp\cap\dots \cap e_{k-1}^\perp$. The set $\{e_1,\dots e_{k-1}\}$ extends to an orthonormal basis $\{e_1,\dots,e_{k-1},e'_k,\dots ,e'_n\}$ of $\Rn$. Thus, $H=\textnormal{span}\{e'_k,\dots,e'_n\}$ and $H^\perp=\textnormal{span}\{e_1,\dots,e_{k-1}\}$. Since $\dim H^\perp\le n-2$, we know from Fact 1 that there exists a unit vector $e_k$, such that $e_K^\perp\supseteq H^\perp$ and $K$ is symmetric with respect to $e_k^\perp$. But then $e_k\in H$, thus $e_k$ is orthogonal to $e_1,\dots,e_{k-1}$ and Fact 2 is proved.

Fact 1 shows immediately that there exists a map $G_{n,n-2}\ni V\mapsto e_V\in \Sn$ (where $G_{n,n-2}$ is the set of all $(n-2)$-dimensional subspaces of $\Rn$), such that $V\subseteq e_V^\perp$ and $K$ is symmetric with respect to $e_V^\perp$, for all $V\in G_{n,n-2}$. Let, also, $e_1,\dots,e_{n-1}$ be the unit vectors constructed in Fact 2 (for $k=n-1$) and let $e_n\in e_1^\perp\cap\dots\cap e_{n-1}^\perp\cap \Sn$.

\emph{Case I. The vector $e_V$ is orthogonal to $e_n$, for all $(n-2)$-dimensional subspaces $V$ of $\Rn$.}\\
In this case, let $e_0$ be any unit vector which is orthogonal to $e_n$ and let $V$ be any $(n-2)$-dimensional subspace of $e_0^\perp$ that does not contain $e_n$. According to our assumption, $e_n\in e_V^\perp$ and, therefore (since $V\subseteq e_V^\perp$), it holds $e_V^\perp=e_0^\perp$. In other words, $K$ is symmetric with respect to any hyperplane that contains the line $\R e_n$. This shows that $K$ is symmetric with respect to the line $\R e_n$.

\emph{Case II. There exists an $(n-2)$-dimensional subspace $V$ of $\Rn$, such that $e_V$ is not orthogonal to $e_n$.} \\
Then, the vectors $e_1,\dots,e_{n-1},e_V$ are linearly independent and $K$ is symmetric with respect to the hyperplanes $e_1^\perp,\dots,e_{n-1}^\perp,e_V^\perp$. The latter shows that 
$$\int_K\langle x,e_1\rangle dx=\dots =\int_K\langle x,e_{n-1}\rangle dx=\int_K\langle x,e_V\rangle dx =0.$$ 
Since $e_1,\dots,e_{n-1},e_V$ are linearly independent, it follows that $$\int_K\langle x,y\rangle dx=0,\qquad\textnormal{for all }y\in \Rn.$$This shows that $b(K)=o$ and, therefore, $K$ is a Euclidean ball centered at the origin. In particular, $K$ is symmetric with respect to some (any) straight line through the origin.
$\square$
\\

The fact that any solution $K$ to (\ref{eq-thm-main}) is axially symmetric implies a regularity property for the boundary of $K$ that will be used subsequently.

\begin{corollary}\label{cor-regularity}
Let $K$ be a convex body in $\Rn$, that solves (\ref{eq-thm-main}) for some $G\in{\cal A}(n)$. Then, $K$ is a regular strictly convex body. In particular, $h_K$ is of class $C^1$.
\end{corollary}
\begin{proof}
We know that $S_K$ is absolutely continuous with respect to $\Ha$. Moreover, if $\omega$ is a Borel subset of $\Sn$, the identity
$$S_K(\omega)=\int_\omega dS_K=\int_\omega G(h_K(v))dS_K(v),$$
together with the fact that $G$ is positive, shows that \begin{equation}\label{eq-S_K(omega)}
S_K(\omega)>0,\qquad\textnormal{if \ \ }\Ha(\omega)>0.
\end{equation}
If $n=2$, then a segment on the boundary of $K$ is a facet of $K$, a fact that contradicts the absolute continuity of $S_K$. If $n\geq 3$, since $K$ is a body of revolution, the strict convexity of $K$ is proved in \cite[Lemma 8.7]{petty_1985}. It remains to prove that $K$ is regular. Let $T$ be a 2-dimensional profile of the body of revolution $K$ (set $T=K$ if $n=2$) and let $E$ be the 2-dimensional subspace of $\Rn$ spanned by $T$. If $K$ is not regular, then $T$ is also not regular, hence there exists a point $x\in\bd T$ and a non-empty open set $N$ in $\Sn\cap E$, such that $\nu$ supports $T$ at $x$, for all $\nu\in N$. If $n=2$, this already shows that $S_K(N)=0$, which contradicts (\ref{eq-S_K(omega)}). If $n\geq 3$, denote by $\overline{N}$ and $X$ the subsets of $\Sn$ and $\bd K$ respectively, obtained by rotating $N$ and $x$ about the axis of symmetry of $K$, respectively. Then, $\overline{N}$ is non-empty and open in $\Sn$ and each vector $v\in\overline{N}$ supports $K$ at some point of $X$. This shows that $S_K(\overline{N})=\Ha(X)=0$, which is again a contradiction.  
\end{proof}
It should be remarked that since $G$ is strictly positive, Corollary \ref{cor-regularity} can also be deduced from a general regularity theorem due to Caffarelli \cite{cafferelli_1991} for Monge-Amp\`{e}re equations (see also \cite{figalli_2018}), without making use of Proposition \ref{lemma-key}. We choose to avoid using such a deep result.
\section{Gluing lemmas}
\hspace*{1.5em}The main goal of this section is to establish a preparatory step (Lemma \ref{lemma-technical} below) towards the proof of the remaining part of Theorem \ref{thm-main} that will allow us to deduce that if $K$ is a non-spherical convex body that solves (\ref{eq-thm-main}), then there exists a non-spherical centrally symmetric convex body $\overline{K}$ that approximately solves (\ref{eq-thm-main}). This will be done by gluing together part of the boundary of $K$, part of the boundary of $-K$ and appropriate spherical parts. Since the barycentre of any centrally symmetric convex body is the origin, it will then not be hard to finish the proof of Theorem \ref{thm-main}. Let us first fix some notation that will be used until the end of this note.

Let $e_1=(1,0,\dots,0)$ and $e_2=(0,1,0,\dots,0)$ be the first two vectors of the standard orthonormal basis in $\Rn$ and $K$ be a convex body which is symmetric with respect to the $x_2$-axis (i.e. the line $\R e_2$). We denote by $E$, the subspace spanned by $e_1$ and $e_2$, which we identify with $\R^2$. We set $Q(E)\subseteq \R^2$ to be the planar convex body $K\cap E$. 

If $T$ is a strictly convex body in $\R^2$ and $\nu\in\mathbb{S}^1$, $p_T(\nu)$ will stand for the (unique) point in $\bd T$ such that the supporting line of $T$, whose outer unit normal vector is $\nu$, touches $T$ at $p_T(\nu)$. Notice that if $T$ is also assumed to be regular, then $p_T=\eta^{-1}_T$ and, therefore, $p_T$ is a continuous map. 

We will need the following observation.     
 
\begin{lemma}\label{prop-gluing}
Let $T_1,T_2$ be regular strictly convex bodies in $\R^2$ and let $p,\ q$ be two distinct intersection points of $\bd T_1$ and $\bd T_2$. Let $l$ be the straight line through $p$ and $q$ and let $l^+$, $l^-$ be the two closed half-planes defined by $l$. Assume, furthermore, that the tangent lines of $T_1$ and $T_2$ respectively, coincide. Then, the set $T:=(T_1\cap l^+)\cup (T_2\cap l^-)$ is also a regular strictly convex body.
\end{lemma}
\begin{proof}
Clearly, for every point $x\in\bd T$, the tangent line $l^T_x$ of $\bd T= ((\bd T_1)\cap l^+)\cup (\bd T_2\cap l^-)$ exists (this is trivially true if $x\neq p,q$, while if $x=p$ (and similarly if $x=q$) it follows from the fact that the tangent lines of $T_1$ and $T_2$ at $p$ coincide). Thus, it suffices to show that for each point $x\in \bd T$, it holds $l^T_x\cap \inter T=\emptyset$.

Let $x\in \bd T$. We may assume that $x\in (\bd T_1)\cap l^+$. Let $l_x=l_x^{T_1}$ be the tangent line of $T_1$ at $x$. Denote by $l_x^-$ the closed half-plane defined by $l_x$, that contains $T_1$ and set $l'_x:=\R^2\setminus l_x^-$ (the interior of the other closed half-plane). If $x=p$ or $x=q$, our assumption implies immediately that $l_x\cap \inter T=\emptyset$. Therefore, we may assume that $x\neq p,q$. Since it is trivially true that $l_x\cap \inter T_1=\emptyset$, it suffices to show that $l_x\cap T_2\cap l^-=\emptyset$.

Set $A$ to be the closed convex set $l^+\cap l_p^-\cap l_q^-$. The line segment $[p,q]$ is obviously one of the three sides (possibly unbounded) of $A$. Since $l_x$ is  disjoint from $[p,q]$, it follows that $A\setminus l_x$ is contained in $l_p'\cap l_q'$. However, $T_2\cap l^-$ is contained in $l_p^-\cap l_q^-\cap l^-\subseteq (\R^2\setminus A)\cup [p,q]$, thus $l_x\cap T_2\cap l^-$ is empty, as claimed.
\end{proof}
Now we are ready to state and prove the main technical fact of this section.
\begin{lemma}\label{lemma-technical}
Let $K$ be a regular strictly convex body in $\Rn$ (in particular, $S_K$ is absolutely continuous with respect to $\Ha$), which is symmetric with respect to the $x_2$-axis and let $B_1:=r_1B_2^n$, $B_2:=r_2B_2^n$, for some $0<r_1<r_2$. Assume, furthermore, that for some open set $U\subseteq \Sn$, $K$ satisfies
\begin{equation}\label{eq-lemma-technical-1}
f_K(v)=G(h_K(v)),\qquad\textnormal{for almost every }v\in\Sn\setminus U,
\end{equation}for some $G\in{\cal A}(n)$.
If $T:=Q(K)$, then the following are true.
\begin{enumerate}[i)]
\item If $p_T(1,0)$ lies on the $x_1$-axis, then there exists a centrally symmetric strictly convex body $\overline{K}$ in $\Rn$, such that $K\cap \{x\in \Rn:x_2\geq 0\}= \overline{K}\cap \{x\in \Rn:x_2\geq 0\}$, satisfying
\begin{equation}\label{eq-lemma-technical-2}
f_{\overline{K}}(v)=G(h_{\overline{K}}(v)),\qquad\textnormal{for almost every }v\in \Sn\setminus (U\cup U'),
\end{equation}
where $U'\subseteq \Sn$ is the reflection of $U$ with respect to the hyperplane $e_2^\perp$.
\item Assume that $U=\emptyset$ and that there exist two distinct vectors $\nu_1,\nu_2\in\Sen\cap \R_+^2$, such that $p_T(\nu_i)=p_{Q(B_i)}(\nu_i)$, $i=1,2$. Assume, furthermore, that
\begin{description}
\item [a)] $r_1<h_{T}(\nu)<r_2$, for all $\nu\in \Sen$ with $\min\{\langle\nu_1,e_2\rangle,\langle\nu_2,e_2\rangle\}\leq \langle\nu,e_2\rangle\leq \max\{\langle\nu_1,e_2\rangle,\langle\nu_2,e_2\rangle\}$.
\item [b)] $r_1^{n-1}=f_{B_1}(\nu_1)\leq G(r_1)$ and $r_2^{n-1}=f_{B_2}(\nu_2)\geq G(r_2)$.
\end{description}
Then, there exists a family of functions $\{G_\varepsilon\}_{\varepsilon\in(0,\varepsilon_0)}\subseteq {\cal A}(n)$, which is uniformly bounded on $(0,\varepsilon_0)\times [r_1,r_2]$ and a centrally symmetric strictly convex body $\overline{K}$ in $\Rn$, which is not an ellipsoid, satisfying
\begin{equation}\label{eq-lemma-technical-3} 
f_{\overline{K}}(v)=G_\varepsilon(h_{\overline{K}}(v)),\qquad\textnormal{for almost every }v\in \Sn\setminus U_\varepsilon,
\end{equation}for all $\varepsilon\in(0,\varepsilon_0)$, where $\{U_\varepsilon\}$ is a family of open sets in $\Sn$ such that $\Ha(U_\varepsilon)\xrightarrow{\varepsilon\to 0^+}0$.
\end{enumerate}
\end{lemma}
\begin{proof} (i) Since $p_T(1,0)$ lies on the $x_1$-axis and since $T$ is symmetric with respect to the $x_2$-axis, it holds $p_T(1,0)=p_{-T}(0,1)$ and $p_T(-1,0)=p_{-T}(-1,0)$. Thus, by Lemma \ref{prop-gluing}, the set $\overline{T}:=\left(T\cap\{x\in\R^2:x_2\ge 0\}\right)\cup \left(-T\cap \{x\in\R^2:x_2\le 0\}\right)$ is a (centrally symmetric) strictly convex body. Let $\overline{K}$ be the centrally symmetric strictly convex body in $\Rn$ obtained by rotating $\overline{T}$ about the $x_2$-axis (this is if $n\ge 3$, otherwise we set $\overline{K}:=\overline{T}$). Then, $K\cap \{x\in \Rn:x_2\geq 0\}= \overline{K}\cap \{x\in \Rn:x_2\geq 0\}$. Moreover, it holds
$$h_{\overline{K}}(v)=\begin{cases}h_K(v),& \langle v,e_2\rangle\geq 0\\
h_K(-v),& \langle v,e_2\rangle\leq 0\end{cases},\qquad v\in\Sn$$
and by (\ref{eq-F_K}), it follows that 
$$f_{\overline{K}}(v)=\begin{cases}f_K(v),& \langle v,e_2\rangle\geq 0\\
f_K(-v),& \langle v,e_2\rangle< 0\end{cases},\qquad \textnormal{for almost every }v\in\Sn.$$
The formulae written above for $h_{\overline{K}}$ and $f_{\overline{K}}$, together with (\ref{eq-lemma-technical-1}), immediately give (\ref{eq-lemma-technical-2}). \\
(ii) Set $\nu_1':=\nu_1$, $\nu_2':=\nu_2$, $B_1':=B_1$, $B_2'=B_2$, if $\langle \nu_1,e_2\rangle<\langle \nu_2,e_2\rangle$ and $\nu_1':=\nu_2$, $\nu_2':=\nu_1$, $B_1':=B_2$, $B_2'=B_1$, otherwise. Let, also, $D_i:=Q(B'_i)$, $i=1,2$. Moreover, set $\nu_i''$ to be the reflection of $\nu_i'$ with respect to the $x_2$-axis, $i=1,2$. Finally, let $l_i=\{x\in\R^2:x_2=t_i\}$ to be the line through the points $p_T(\nu_i')$ and $p_{T}(\nu_i'')$, for some $t_i\geq 0$ and let $l_i^+:=\{x\in\R^2:x_2\geq t_i\}$, $l_i^+:=\{x\in\R^2:x_2\leq t_i\}$, $i=1,2$. Applying Lemma \ref{prop-gluing} twice, we conclude that the set 
$$T':=(D_2\cap l_2^+)\cup (T\cap l_2^-\cap l_1^+)\cup (D_1\cap l_1^-)$$is a regular strictly convex body in $\R^2$ that contains the origin in its interior. If $K'$ is the (regular strictly convex) body of revolution in $\Rn$ (if $n=2$, set again $K':=T'$), obtained by rotating $T'$ about the $x_2$-axis, then as in the proof of part (i), we find
$$h_{K'}(v)=\begin{cases}
h_{B_1'}(v),& \langle v,e_2\rangle \leq \langle \nu_1',e_2\rangle\\
h_{K}(v),& \langle \nu_1',e_2\rangle \leq \langle v,e_2\rangle\leq \langle \nu_2',e_2\rangle\\
h_{B_2'}(v),& \langle v,e_2\rangle \geq \langle \nu_2',e_2\rangle
\end{cases},\qquad v\in\Sn$$ 
and
$$f_{K'}(v)=\begin{cases}
f_{B_1'}(v),& \langle v,e_2\rangle \leq \langle \nu_1',e_2\rangle\\
f_{K}(v),& \langle \nu_1',e_2\rangle <\langle v,e_2\rangle\leq \langle \nu_2',e_2\rangle\\
h_{B_2'}(v),& \langle v,e_2\rangle > \langle \nu_2',e_2\rangle
\end{cases},\qquad \textnormal{for almost every }v\in\Sn.$$Notice that the formulae above are still valid even if $\nu_2'=e_2$. 

Because of assumption (b), by Lemma \ref{lemma-facts-G_varepsilon}, there exists a uniformly bounded (in $(0,\varepsilon_0)\times[r_1,r_2]$) family of functions $\{G_\varepsilon\}_{\varepsilon\in(0,\varepsilon_0)}\subseteq {\cal A}(n)$, such that $G_\varepsilon(r_i)=r_i^{n-1}$, $i=1,2$ and for all $\varepsilon\in(0,\varepsilon_0)$, it holds $G_\varepsilon(\theta)=G(\theta)$, for all $\theta\in [r_1+\varepsilon,r_2-\varepsilon]$. Using assumption (a), (\ref{eq-lemma-technical-1}) and the representation for $h_{K'}$ and $f_{K'}$ derived above, we immediately conclude that 
$$f_{K'}(v)=G_\varepsilon(h_{K'}(v)),\qquad\textnormal{for almost all }v\in\Sn\setminus V_\varepsilon,$$
where $V_\varepsilon:=\{v\in\Sn:r_1<h_{K'}(v)<r_1+\varepsilon\}\cup\{v\in\Sn:r_2-\varepsilon<h_{K'}(v)<r_2\}$. Notice that $V_\varepsilon\searrow \emptyset$ and, therefore, $\Ha(V_\varepsilon)\xrightarrow{\varepsilon\to 0^+} 0$. Since $p_{Q(K')}(1,0)=p_{D_1}(1,0)$ lies on the $x_1$-axis, the existence of a non-spherical (this follows from assumption (a)) centrally symmetric strictly convex body $\overline{K}$ satisfying (\ref{eq-lemma-technical-3}) follows from part (i). Additionally, $K'$ and hence $\overline{K}$ contains a spherical part on its boundary (but is not spherical itself) and, consequently, $\overline{K}$ cannot be an ellipsoid.
\end{proof}
\section{Further restriction of the class of candidates and proof of Theorem \ref{thm-main}}
\hspace*{1.5em}The proof of the remaining part of Theorem \ref{thm-main} will follow from the next proposition (which actually provides some extra information about solutions of (\ref{eq-thm-main}), in the case that $G$ is not necessarily strictly monotone). Before we state it, it will be convenient to introduce a standard reparametrization of $h_T$, where $T$ is a planar convex body. More specifically, we set
$$\overline{h}_T(\varphi):=h_T(\cos \varphi,\sin\varphi),\qquad \varphi\in \R.$$
\begin{proposition}\label{prop-more}
Let $K$ be a convex body in $\Rn$ that contains the origin in its interior, solves (\ref{eq-thm-main}) for some function $G\in {\cal A}(n)$ and is symmetric with respect to the $x_2$-axis. Then, the following hold.
\begin{enumerate}[i)]
\item $\overline{h}_{Q(K)}$ is monotone on $[-\pi/2,\pi/2]$.
\item If $K$ is not a Euclidean ball centered at the origin, then $\overline{h}_{Q(K)}$ is strictly monotone in a neighbourhood of $\varphi=0$.
\end{enumerate}
\end{proposition}
\noindent Proof of Theorem \ref{thm-main}.\\
It remains to assume that $G$ is monotone. It clearly suffices to show that $G(h_K)$ is constant on $\Sn$. By Proposition \ref{lemma-key}, we may assume that any solution $K$ to (\ref{eq-thm-main}) is symmetric with respect to the $x_2$-axis. Then, by Proposition \ref{prop-more}, $G(\overline{h}_K)$ is monotone. It follows immediately that for $u=e_2$ or $u=-e_2$, it holds $G(h_K(v))\leq G(h_K(v'))$, for any two vectors $v,v'\in\Sn$ with $\langle v,u\rangle\leq 0\leq \langle v',u\rangle$. This shows in particular that  
$$\int_{\Sn\cap \{x_2\le 0\}}-\langle v,u\rangle G(h_K(v))d\Ha(v)\leq\int_{\Sn\cap \{x_2\ge 0\}}\langle v',u\rangle G(h_K(v))d\Ha(v'),$$with equality if and only if $G(h_K)$ is constant on $\Sn$.
On the other hand, since the barycentre of $S_K$ is always at the origin, equality  must hold in the previous inequality and Theorem \ref{thm-main} is proved. $\square$ 
\\

Before we prove Proposition \ref{prop-more}, we will need some additional considerations.

\begin{proposition}\label{prop-appendix}
Let $L$, $M$ be regular strictly convex bodies in $\Rn$, such that $S_L=f_Ld\Ha$ and $S_M=f_Md\Ha$, for some continuous functions $f_L,f_M:\Sn\to(0,\infty)$. Assume, furthermore, that there is a point $p\in \bd L\cap \bd M$ and a neighbourhood $V$ of $p$ in $\bd L$, such that $V$ is contained in $M$. Let $\nu$ be the unit normal vector of $L$ (and therefore of $M$) at $p$. If at least one of $L$ and $M$ is of class $C^2_+$, then it holds $f_L(\nu)\leq f_M(\nu)$.\end{proposition}
Proposition \ref{prop-appendix} is trivial if both convex bodies $L$ and $M$ are of class $C_+^2$. However, if we assume $C_+^2$ regularity for both $L$ and $M$, then as far as we can tell, Proposition \ref{prop-more} (and therefore the final part of Theorem \ref{thm-main}) can only be proved under the additional assumption that the boundary of $K$ is $C^2$. The interested reader can find a full proof of Proposition \ref{prop-appendix} in the Appendix of this note.
\\
\\
Proof of Proposition \ref{prop-appendix} in the case that $L$ and $M$ are of class $C^2_+$.\\ 
Since both $L$ and $M$ are both of class $C_+^2$, their boundaries are of class $C^2$ and it is well known that the assumption of Proposition \ref{prop-appendix} implies that the curvature of $\bd L$ at $p$ dominates the curvature of $\bd M$ at $p$. Thus (since $\eta_L(p)=\eta_M(p)=\nu)$, (\ref{eq-curvature}) yields that $f_L(\nu)\le f_M(\nu)$. $\square$
\begin{lemma}\label{lemma-local-extr} Let $K$ be a strictly convex body, which is symmetric with respect to the $x_2$-axis and contains the origin in its interior. 
The following are true.
\begin{enumerate}[i)]
\item If for some $\varphi_0\in\R$ the function $\overline{h}_{Q(K)}$ attains a local extremum at $\varphi_0$, then it holds
\begin{equation}\label{eq-lemma-local-extr}
p_{Q(K)}(\cos\varphi_0,\sin\varphi_0)=\overline{h}_{Q(K)}(\varphi_0)(\cos\varphi_0,\sin\varphi_0),
\end{equation}while if (\ref{eq-lemma-local-extr}) does not hold, then $\overline{h}_{Q(K)}$ is strictly monotone in a neighbourhood of $\varphi_0$.
\item If $K$ solves (\ref{eq-thm-main}) for some function $G\in{\cal A}(n)$ and if $\varphi_0$ is a point of local minimum (resp. local maximum) for $\overline{h}_{Q(K)}$, then we have $G(\overline{h}_{Q(K)}(\varphi_0))\geq \overline{h}_{Q(K)}(\varphi_0)^{n-1}$ (resp. $G(\overline{h}_{Q(K)}(\varphi_0))\leq \overline{h}_{Q(K)}(\varphi_0)^{n-1}$).\end{enumerate}
\end{lemma}
\begin{proof} Notice that (\ref{eq-lemma-local-extr}) holds if and only if the position vector of $p_{Q(K)}(\cos\varphi_0,\sin\varphi_0)$ is orthogonal to the vector $(-\sin\varphi_0,\cos\varphi_0)$.  On the other hand, (\ref{eq-gauss-map}) and the fact that $p_{Q(K)}=\eta^{-1}_K$ give
$$\overline{h}_{Q(K)}'(\varphi_0)=\langle \nabla h_{Q(K)}(\cos\varphi_0,\sin\varphi_0),(-\sin\varphi_0,\cos\varphi_0)\rangle=\langle p_{Q(K)}(\cos\varphi_0,\sin\varphi_0),(-\sin\varphi_0,\cos\varphi_0)\rangle.$$
This proves (i).

To prove the second part, recall that $K$ is regular (as Corollary \ref{cor-regularity} shows) and set $D:=\overline{h}_{Q(K)}(\varphi_0)B_2^2\subseteq \R^2$.  If $\overline{h}_{Q(K)}$ attains a local minimum at $\varphi_0$, then in a neighbourhood of $\varphi_0$ it holds $h_{Q(K)}(\varphi)\geq h_D(\varphi)$. However, (\ref{eq-lemma-local-extr}) shows that $p_{Q(K)}(\cos\varphi_0,\sin\varphi_0)=p_{D}(\cos\varphi_0,\sin\varphi_0)$. It follows that a neighbourhood of $p_{D}(\cos\varphi_0,\sin\varphi_0)$ in the boundary of $D$ is contained in $Q(K)$. Thus, we immediately see that if $B:=\overline{h}_{Q(K)}(\varphi_0)B_2^n$, then the boundary of $B$ touches the boundary of $K$ at the point $p:=(p_D(\cos\varphi_0,\sin\varphi_0),0,\dots,0)\in\Rn$ and some neighbourhood of $p$ in the boundary of $B$ is contained in $K$. Since the outer unit normal vector for both $K$ and $B$ at $p$ equals $\nu:=(\cos\varphi_0,\sin\varphi_0,0,\dots,0)$, Proposition \ref{prop-appendix} shows that 
$$f_K(\nu)\geq f_B(\nu)=\overline{h}_{Q(K)}(\varphi_0)^{n-1}$$
or equivalently,
$$G(\overline{h}_K(\varphi_0))=G(h_K(\nu))\geq\overline{h}_{Q(K)}(\varphi_0)^{n-1}.$$
The remaining case (i.e. when $\overline{h}_{Q(K)}$ attains a local maximum at $\varphi_0$) can be treated similarly.
\end{proof}
\noindent
Proof of Proposition \ref{prop-more}.\\
Set $T:=Q(K)$. Assume that (ii) is not true. Then, according to Lemma \ref{lemma-local-extr}, $p_T(1,0)$ must be contained in the $x_1$-axis.  Then, by Lemma \ref{lemma-technical} (i), there exists a centrally symmetric convex body $\overline{K}$ with $\overline{K}\cap \{x\in\Rn:x_2\geq 0\}=K\cap\{x\in\Rn:x_2\geq 0\}$, satisfying (\ref{eq-thm-main}). Since $b(\overline{K})=o$, Proposition \ref{lemma-key} shows that $\overline{K}=\lambda_1B_2^n$, for some $\lambda_1>0$. Thus, $K\cap\{x\in\Rn:x_2\geq 0\}=\lambda_1B_2^n\cap\{x\in\Rn:x_2\geq 0\}$. Putting $-K$ in the place of $K$, we conclude that $-K\cap\{x\in\Rn:x_2\geq 0\}=\lambda_2B_2^n\cap\{x\in\Rn:x_2\geq 0\}$, for some $\lambda_2>0$. This shows that $K=-K=\lambda_1B_2^n=\lambda_2B_2^n$ and (ii) is proved.

To prove (i), it suffices to show that $\overline{h}_T$ is monotone on $[0,\pi/2]$. Indeed, if this is true, one can replace $K$ by $-K$ to show that $\overline{h}_T$ is also monotone on $[-\pi/2,0]$. Taking (ii) into account, this would imply that $\overline{h}_T$ is monotone on $[-\pi/2,\pi/2]$. 

Let $\varphi_0$ be such that the interval $[0,\varphi_0]$ is the largest interval of the form $[0,\varphi]\subseteq [0,\pi/2]$, on which $\overline{h}_T$ is monotone. We need to show that $\varphi_0=\pi/2$. Assume that this is not true. Then, $\overline{h}_T$ is not constant on $[\varphi_0,\pi/2]$, so we can find a minimal (with respect to inclusion) interval $[\varphi_1,\varphi_2]\subseteq [\varphi_0,\pi/2]$, such that the extrema $\min_{\varphi\in[\varphi_0,\pi/2]}\overline{h}_T(\varphi)$, $\max_{\varphi\in[\varphi_0,\pi/2]}\overline{h}_T(\varphi)$ are attained at the points $\varphi_1$, $\varphi_2$. In particular, all values of $\overline{h}_T$ on $(\varphi_1,\varphi_2)$ lie strictly between $\overline{h}_T(\varphi_1)$ and $\overline{h}_T(\varphi_2)$. Moreover, $\varphi_1$ and $\varphi_2$ are points of local extremum for $\overline{h}_T$ (notice that since $\overline{h}_T$ is symmetric about the point $\varphi=\pi/2$, this is also true even if $\phi_2=\pi/2$). It follows from Lemma \ref{lemma-local-extr} that if $r_i=\overline{h}_T(\varphi_i)$, then $p_T(\cos\varphi_i,\sin \varphi_i)=p_{Q(r_iB_2^n)}(\cos\varphi_i,\sin\varphi_i)$, $i=1,2$. Notice, furthermore, that if $\overline{h}_T$ attains a local minimum at $\varphi_1$, then $\overline{h}_T$ attains a local maximum at $\varphi_2$ and vice versa. Using again Lemma \ref{lemma-local-extr}, we see that in the first case we have $r_1^{n-1}\leq G(r_1)$ and $r_2^{n-1}\geq G(r_2)$, while in the second case the inequalities are reversed. In any case, the assumptions of Lemma \ref{lemma-technical} (ii) are satisfied with $c_1=\min\{r_1,r_2\}$, $c_2=\max\{r_1,r_2\}$ and $a_i=c_i^{n-1}$, $i=1,2$. Let $\overline{K}$, $\{G_\varepsilon\}_{\varepsilon\in(0,\varepsilon_0)}$, $\{U_\varepsilon\}$ be as in the statement of Lemma \ref{lemma-technical}. Then, it holds
\begin{equation}\label{eq-prop-more}
f_{\overline{K}}(v)=G_\varepsilon(h_{\overline{K}}(v)), \qquad\textnormal{for almost every }v\in\Sn\setminus U_\varepsilon.
\end{equation}

Since $\overline{K}$ is centrally symmetric but not an ellipsoid, by Lemma \ref{lemma-before-key} (ii), there exists an orthogonal map $O$, such that if we set $K_1:=O\overline{K}$, it holds
$$\int_{\Sn}h'_{K_1}(v)dS_{K_1}<0.$$On the other hand, (\ref{eq-prop-more}) and Lemma \ref{lemma-eq-l-st-s-der} show that \\
$\displaystyle{\int_{\Sn}h'_{K_1}(v)dS_{K_1}(v)}$
\begin{eqnarray*}
&=&\int_{\Sn}G_\varepsilon(h_{K_1}(v))h'_{K_1}(v)d\Ha(v)+\int_{U_\varepsilon}[f_{K_1}(v)-G_\varepsilon(h_{K_1}(v))]h'_{K_1}(v)d\Ha(v)\\
&\geq&\int_{U_\varepsilon}[f_{K_1}(v)-G_\varepsilon(h_{K_1}(v)]h'_{K_1}(v)d\Ha(v)\xrightarrow{\varepsilon\to 0^+}0,
\end{eqnarray*}
where we used the fact that $h'_{K_1}$ is bounded (see Corollary \ref{prop-diff}) and, therefore, the last integral indeed exists and converges to zero. We arrived at a contradiction because we assumed that $\varphi_0<\pi/2$. Hence, $\overline{h}_T$ is monotone on $[0,\pi/2]$, as claimed.  
$\square$
\section{Proof of Theorem \ref{thm-counterexample}}
\hspace*{1em}We may assume that in the statement of Theorem \ref{thm-counterexample}, $a=\lambda e_2$ for some $\lambda>0$. Fix a number $r>2\lambda$. Then, it is clear that the ball $rB_2^n+\lambda e_2$ contains the origin in its interior and that the function $\overline{h}_{Q(rB_2^n+\lambda e_2)}(\varphi)=\overline{h}_{rB_2^2+\lambda e_2}(\varphi)=r+\lambda\sin\varphi$, $\varphi\in\R$, is strictly increasing on $[-\pi/2,\pi/2]$. Let $g:[-1,1]\to[0,\infty)$ be an even $C^\infty$ function supported on $[-1/2,1/2]$. Since the ball $rB_2^n$ is uniformly convex, it follows that if $m\in\mathbb{N}$ is large enough, them the function $\Sn\ni v\mapsto r+(1/m)g(\langle v,e_2\rangle)$ is the support function of a centrally symmetric strictly convex body $K_m$ in $\Rn$, which is symmetric with respect to the $x_2$-axis and has $C^\infty$ boundary. Notice that $\{h_{K_m}\}$ converges uniformly to $h_{rB_2^n}\equiv r$, while all derivatives (on the sphere) of $\{h_{K_m}\}$ converge uniformly to 0. In particular, if we define $\overline{f}_{M}:[-\pi/2,\pi/2]\to\R$ by
$$\overline{f}_M(\varphi):=f_M(\cos\varphi,\sin\varphi,0,\dots,0),$$
where $M=K_m$ or $rB_2^2$, we conclude that 
\begin{equation}\label{eq-counterexample-uniform}
\overline{f}_{K_m}\to\overline{f}_{rB_2^n}\equiv r^{n-1}\qquad\textnormal{and}\qquad \overline{f}_{K_m}'\to\overline{f}_{rB_2^n}'\equiv 0,
\end{equation}uniformly of $[-\pi/2,\pi/2]$. Observe, also, that 
\begin{equation}\label{eq-counterexample-0}
\overline{f}_{K_m}(\varphi)=r,\qquad\textnormal{for all }\varphi\in [-\pi/2,-\pi/6)\cup(\pi/6,\pi/2].
\end{equation}

Next, consider the function $\overline{h}_{Q(K_m+\lambda e_2)}(\varphi)=r+(1/m)g(\sin\varphi)+\lambda\sin\varphi$, $\varphi\in\R$. We have, 
\begin{equation}\label{eq-counterexample-der}
\overline{h}_{Q(K_m+\lambda e_2)}'(\varphi)=\cos\varphi\left(\lambda+\frac{1}{m}g'(\sin\varphi)\right),\qquad \varphi\in\R
\end{equation}and, therefore, $\overline{h}_{Q(K_m+\lambda e_2)}$ is strictly increasing on $[-\pi/2,\pi/2]$, provided that $m$ is large enough.
For large $m$, consider $\zeta,\zeta_m:[r-\lambda,r+\lambda]\to[-\pi/2,\pi/2]$ to be the inverses of $\overline{h}_{Q(B_2^n+\lambda e_2)}|_{[-\pi/2,\pi/2]}$ and $\overline{h}_{Q(K_m+\lambda e_2)}|_{[-\pi/2,\pi/2]}$, respectively and set 
\begin{equation*}\label{eq-counterexample-G_m}
G_m(\theta):=\overline{f}_{K_m}(\zeta_m(\theta)),\qquad \theta\in[-\pi/2,\pi/2].
\end{equation*}
One can prove easily that $\{\zeta_m\}$ converges uniformly to $\zeta$ on $[r-\lambda,r+\lambda]$ and, hence, $\{G_m\}$ converges uniformly to the constant $r^{n-1}$ on $[r-\lambda,r+\lambda]$. Furthermore, as (\ref{eq-counterexample-der}) shows, $\zeta_m'$ is bounded on $[r-\lambda/2,r+\lambda/2]$ (if $m$ is large), so from (\ref{eq-counterexample-uniform}) and (\ref{eq-counterexample-0}), we immediately see that $\{G_m'\}$ converges uniformly to 0 on $[r-\lambda,r+\lambda]$. It follows that the sequence $\{\theta G'_m(\theta)+(n+1)G_m(\theta)\}$ converges uniformly to $(n+1)r^{n-1}$ on $[r-\lambda,r+\lambda]$. Consequently, there exists a positive integer $m_0$, such that for any $m\geq m_0$, $G_m$ is (or more precisely, can be extended to; see Remark \ref{remark-we-may-assume}) a member of class ${\cal A}(n)$. However, from the definition of $G_m$, we clearly see that for $m\geq m_0$, it holds 
$$f_{K_m+\lambda e_2}(v)=f_{K_m}(v)=G_m(h_{K_m+\lambda e_2}(v)),\qquad\textnormal{for all }v\in\Sn$$
and the proof of Theorem \ref{thm-counterexample} is complete.
$\square$
\begin{remark}\label{remark-end} It is clear that in the construction above, for any positive integer $k$, $G$ can be chosen to be of class $C^k$. Currently, we do not have an example of a non-spherical convex body $K$ that satisfies (\ref{eq-thm-main}) for some $C^\infty$ function $G\in{\cal A}(n)$.
\end{remark}
\appendix\section{Appendix}
\hspace*{1.5em}This Appendix is devoted to the proof of Proposition \ref{prop-appendix} in full generality.
We will need the following fact, which we believe should be well known. Since we were unable to find an explicit reference, we provide a proof here.
\begin{lemma}\label{lemma-appendix}
Let $A>0$. There exists a constant $C=C(A,n)>0$, such that if $L$ is a convex body in $\Rn$, with absolutely continuous surface area measure with respect to $\Ha$ and 
\begin{equation}\label{eq-bounds-on-f_L} 
\frac{1}{A}\leq f_L(v)\leq A,\qquad \textnormal{for almost every }v\in\Sn,
\end{equation}
then there exists $a\in\Rn$, such that 
$$\frac{1}{C}B_2^n+a\subseteq L\subseteq CB_2^n+a.$$
\end{lemma}
\begin{proof} Any constant $C_1,C_2$, etc. that will appear in this proof will denote a positive constant that depends only on $A$ and the dimension $n$. For an $(n-1)$-dimensional subspace $H$ of $\Rn$, we write $V_H(\cdot)$ for the volume functional in $H$. 

Recall the well known fact that if $L_1$, $L_2$ are convex bodies with $S_{L_1}\leq S_{L_2}$, then it holds $V(L_1)\leq V(L_2)$. Using this and (\ref{eq-bounds-on-f_L}), we obtain
\begin{equation}\label{eq-app-V(L)}
\frac{1}{C_1}\leq V(L)\leq C_1,
\end{equation}
for some $C_1>0$. Moreover, the well known formula
$$V_{e^\perp}(L|e^\perp)=\frac{1}{2}\int_{\Sn}|\langle x,e\rangle|dS_L(x),\qquad e\in\Sn$$
immediately shows that 
\begin{equation}\label{eq-app-vol-proj}
\frac{1}{C_2}\leq V_{e^\perp}(L|e^\perp)\leq C_2,
\end{equation}
for all $e\in\Sn$, where $C_2$ is another positive constant. After a suitable translation, we may assume that the maximal volume ellipsoid $E=tSB_2^n$ contained in $L$ is centered at the origin. Here, $S$ denotes a symmetric positive definite matrix of determinant 1 and $t=(V(E)/V(B_2^n))^{1/n}$. Then, (\ref{eq-app-V(L)}) together with the classical theorem of F. John \cite{john_1948}, yields
$$\frac{1}{C_3}B_2^n\subseteq S^{-1}L\subseteq C_3B_2^n,$$
for some constant $C_3>0$. Equivalently, we may write
\begin{equation}\label{eq-app-S}
\frac{1}{C_3}SB_2^n\subseteq L\subseteq C_3SB_2^n. 
\end{equation}
Let $\lambda,\mu$ be the smallest and the largest eigenvalue of $S$ and $e_\lambda,e_\mu$ be the corresponding eigenvectors respectively. Then, (\ref{eq-app-vol-proj}) and (\ref{eq-app-S}) give
$$\frac{1}{C_3^{n-1}}\frac{1}{\lambda}V_{e_\lambda^\perp}(B_2^n|e_\lambda^\perp)=V_{e_\lambda^\perp}((1/C_3)SB_2^n|e_\lambda^\perp)\leq V_{e_\lambda^\perp}(L|e_\lambda^\perp)\leq C_2$$and
$$C_3^{n-1}\frac{1}{\mu}V_{e_\mu^\perp}(B_2^n|e_\mu^\perp)=V_{e_\mu^\perp}(C_3SB_2^n|e_\mu^\perp)\geq V_{e_\mu^\perp}(L|e_\mu^\perp)\geq \frac{1}{C_2}.$$Consequently, if $C_4:=C_2C_3^{n-1}$, then
$1/C_4\leq \lambda\leq \mu\leq C_4$ and hence, using again (\ref{eq-app-S}), we conclude
$$\frac{1}{C_3C_4}B_2^n\subseteq L\subseteq C_3C_4 B_2^n.$$This completes our proof.
\end{proof}
\noindent Proof of Proposition \ref{prop-appendix}.\\ 
We may clearly assume that $o\in \inter L\cap \inter M$. First, let us prove Proposition \ref{prop-appendix} without any regularity assumption on the boundaries of $L$ and $M$, but under the additional assumption that $V\setminus \{p\}\subseteq \inter M$. This, together with the fact that $p\in\inter L\cap \inter M$ is clearly equivalent to:
\begin{enumerate}[i)]
\item $\rho_L(v_0)=\rho_M(v_0)$, where $v_0=p/|p|$
\item $\rho_L(v)<\rho_M(v)$, for all $v\in U:=\{x/|x|:x\in V\}$.
\end{enumerate}
Since $f_L$ and $f_M$ are continuous, there are sequences of strictly positive $C^\infty$ functions $\{\underline{f}_m\}$ and $\{\overline{f}_m\}$, such that $\underline{f}_m\to f_L$ and $\overline{f}_m\to f_M$, uniformly on $\Sn$. By Minkowski's existence and Uniqueness Theorem, for $m\in\mathbb{N}$, there exist uniquely determined up to translation convex bodies $L_m$ and $M_m$, such that $S_{L_m}=\underline{f}_md\Ha$ and $S_{M_m}=\overline{f}_md\Ha$. The sequences $\{\underline{f}_m\}$ and $\{\overline{f}_m\}$ are uniformly bounded from above and uniformly away from zero, therefore after suitable translations, as Lemma \ref{lemma-appendix} shows, the bodies $L_m$ and $M_m$ are contained in and contain a fixed ball. Hence, by taking subsequences, Blaschke's Selection Theorem shows that we may assume that $L_m\to \overline{L}$ and $M_m\to \overline{M}$ in the Hausdorff metric, for some convex bodies $\overline{L}$ and $\overline{M}$. However, $\underline{f}_m\to f_L$ and $\overline{f}_m\to f_M$ uniformly and thus weakly on $\Sn$, so $\overline{L}$ and $\overline{M}$ are translates of $L$ and $M$ respectively. Finally, we may assume that $\overline{L}=L$ and $\overline{M}=M$. Notice, in addition, that since $\underline{f}_m$ and $\overline{f}_m$ are positive $C^\infty$ functions, it follows (see \cite{pogorelov_1971}) that $L_m$ and $M_m$ are all of class $C_+^2$. 

Since $L_m\to L$ and $M_m\to M$ (and since $o\in\inter L_m\cap \inter M_m$ if $m$ is large enough), we conclude that $\rho_{L_m}/\rho_L\to 1$ and $\rho_{M_m}/\rho_M\to 1$, uniformly on $\Sn$. Thus, since $\min_{v\in \bd V}(\rho_M(v)/\rho_L(v))>1$, it follows that if $m$ is large enough, then $$\rho_{M_m}(v)>c\rho_{L_m}(v), \qquad\textnormal{for all }v\in \bd V,$$where $c>1$ is a constant which is independent of $m$. On the other hand, for $0<\varepsilon<c-1$, it holds 
$$\rho_{M_m}(v_0)<(1+\varepsilon)\rho_{L_m}(v_0)<c\rho_{L_m}(v_0).$$
Consequently, there exists $m_0\in\mathbb{N}$, such that for any $m\geq m_0$, the minimum
$$c_m:=\min_{v\in \cl U}\frac{\rho_{M_m}(v)}{\rho_{L_m}(v)}$$ is attained inside $U$. This shows that if $m\geq m_0$, there exists $v_m\in U$, such that $\rho_{c_mL_m}(v_m)=c_m\rho_{L_m}(v_m)=\rho_{M_m}(v_m)$, while $\rho_{c_mL_m}(v)\leq \rho_{M_m}(v)$, for all $v\in U$. Thus, the triple $(c_mL_m, M_m, v_m)$ satisfies assumptions (i) and (ii) imposed previously and, therefore, satisfies the (weaker) assumptions in the statement of Propostion \ref{prop-appendix}. Since $c_mL_m$ and $M_m$ are of class $C^2_+$, we conclude that if $m\ge m_0$, then 
\begin{equation}\label{eq-prop-appendix}c_m^{n-1}\underline{f}_m(\nu_m)=f_{c_mL_m}(\nu_m)\leq \overline{f}_m(\nu_m),
\end{equation}
where $\nu_m:=\eta_{L_m}(p_m)$ and $p_m:=\rho_{L_m}(v_m)v_m$. 

Next, set $W:=\eta_L(V)$ and let $\{\nu_{k_m}\}$ be a subsequence of $\{\nu_m\}$ that converges to some vector $\nu'\in\Sn$. We claim that $\nu'\in \cl W=\eta_L(\cl V)$. To see this, recall that $p_{k_m}$ is the unique point in $\bd L_{k_m}$, such that 
$$\langle p_{k_m},\nu_{k_m}\rangle=h_{L_{k_m}}(\nu_{k_m}).$$By taking a subsequence, we may assume that $p_{k_m}\to q$, for some point $q\in \cl V$. Thus, it holds $\langle q,\nu'\rangle=h_L(\nu')$ and, therefore, $\nu'=\eta_L(q)\in \eta_L(\cl V)=\cl W$. 

Since $\underline{f}_m$ and $\overline{f}_m$ converge uniformly on $\Sn$ and since $c_m\to 1$, we conclude by (\ref{eq-prop-appendix}) that $$f_L(\nu')\le f_M(\nu').$$

Notice that, in the argument described above, one can replace $V$ by any open set $V'\subseteq V$. Having this in mind, consider a sequence $\{V'_l\}$ of open sets in $\bd L$, all contained in $V$, such that $V'_l\searrow \{p\}$. Then, for $l\in \mathbb{N}$, there exists a  vector $\nu'_l\in \cl \eta_L(V'_l)$, such that $f_L(\nu'_l)\le f_M(\nu'_l)$. Since $\cl \eta_L(V'_l)\searrow \eta_L(\{p\})=\{\nu\}$, it follows that $\nu'_l\to\nu$ and consequently, $$f_L(\nu)\le f_M(\nu). $$ 

It remains to remove the extra assumption $V\setminus \{p\}\subseteq \inter M$. Let $L,M,V,p,\nu$ be as in the statement of Proposition \ref{prop-appendix}. At this point we are going to assume that $L$ is of class $C_+^2$ (the case where $M$ is of class $C_+^2$ can be treated completely similarly and is left to the reader). Let $g:\Sn\to\R$ be a $C^2$ function which is strictly positive on $\Sn\setminus\{\nu\}$ and satisfies $g(\nu)=0$. Then, for small positive $t$, the function $h_L-tg$ is also a support function of class $C_+^2$. Set $L_t$ for the $C_+^2$ convex body whose support function equals $h_L-tg$. Then, $h_{L_t}\leq h_L$, thus $L_t$ is contained in $L$.  Furthermore, it holds $h_{L_t}(\nu)=h_L(\nu)$ and, therefore, $L_t$ is supported by the supporting hyperplane of $L$, whose outer unit normal vector is $\nu$. But since $L_t$ is contained in $L$, it follows that (for small $t$) $p\in\bd L_t$. Moreover, $h_{L_t}(a)<h_L(a)$, for all $a\in W\setminus\{\nu\}=\eta_L(V)\setminus\{\nu\}$, so $\eta^{-1}_{L_t}(W)\setminus\{p\}\subseteq \inter L\subseteq \inter M$. This, together with the fact that $p\in \bd L_t\cap \bd M$, shows that $$f_{L_t}(\nu)\leq f_M(\nu),$$for small $t>0$. However, since $L$ is of class $C^2_+$, (\ref{eq-F_K}) shows that $f_{L_t}(\nu)\xrightarrow{t\to 0^+}f_L(\nu)$ and the result follows. $\square$

\vspace{1.5 cm}

\noindent Christos Saroglou \\
Department of Mathematics\\
University of Ioannina\\
Ioannina, Greece, 45110 \\
E-mail address: csaroglou@uoi.gr \ \&\ christos.saroglou@gmail.com

\end{document}